\documentclass
[11pt]{amsart}
\usepackage{amsmath,amsthm, amscd, amssymb, amsfonts}
\usepackage[all]{xy}
\usepackage{hhline}
\usepackage[dvips, dvipsnames, usenames]{color}

\usepackage[active]{srcltx}

\newcommand{\nc}{\newcommand}
\newcommand{\ot}{\otimes}

\newcommand{\co}{\operatorname{co}}
\newcommand{\diag}{\operatorname{diag}}

\newcommand{\ord}{\operatorname{ord}}

\nc{\ydk}{^{K}_{K}\mathcal{YD}}

\nc{\Tr}{\mathrm{Tr}}
\nc{\X}{\mathbf{X}}

\newcommand{\ku}{\mathbb C}

\newcommand{\N}{{\mathbb N}}

\newcommand{\M}{{\mathcal M}}

\newcommand{\D}{{\mathcal D}}
\newcommand{\E}{{\mathcal E}}
\newcommand{\II}{{\mathcal I}}

\nc{\eps}{\varepsilon}

\newcommand{\com}{\Delta}

\newcommand{\End}{\operatorname{End}}

\newcommand\ad{\operatorname{ad}}

\nc{\coM}{\M^\ast(2,\Bbbk)}
\nc{\coMtres}{\M^\ast(3,\Bbbk)}
\nc{\coMcua}{\M^\ast(4,\Bbbk)}
\nc{\coMcin}{\M^\ast(5,\Bbbk)}
\nc{\coMt}{\M^\ast(t,\Bbbk)}
\nc{\coMj}{\M^\ast(j,\Bbbk)}
\nc{\coMn}{\M^\ast(n,\Bbbk)}
\nc{\coMd}{\M^\ast(d,\Bbbk)}
\nc{\Ho}{H_0}
\nc{\GH}{G(H)}
\nc{\mas}{\oplus}
\nc{\cA}{\mathcal{A}}
\nc{\yd}{^{C_2}_{C_2}\mathcal{YD}}
\nc{\PH}{\cP(H)}
\nc{\e}{\varepsilon}
\nc{\GL}{\operatorname{GL}}
\nc{\wact}{\rightharpoonup}
\nc{\cark}{char\,k}
\nc{\adl}{\ad_\ell}
\nc{\cP}{\mathcal{P}}
\nc{\cU}{\mathcal{U}}
\nc{\fD}{\mathfrak{D}}
\nc{\cE}{\mathcal{E}}
\nc{\cS}{\mathcal{S}}
\nc{\be}{\textbf{e}}

\newcommand\id{\operatorname{id}}

\def\pf{\begin{proof}}
\def\epf{\end{proof}}

\theoremstyle{remark}

\numberwithin{equation}{section}
\theoremstyle{plain}

\newtheorem{lema}{Lemma}[section]
\newtheorem{theorem}[lema]{Theorem}
\newtheorem{cor}[lema]{Corollary}

\newtheorem{prop}[lema]{Proposition}
\newtheorem{proposition}[lema]{Proposition}
\newtheorem{example}[lema]{Example}

\theoremstyle{definition}
\newtheorem{definition}[lema]{Definition}

\theoremstyle{remark}
\newtheorem{obs}[lema]{Remark}
\newtheorem{remark}[lema]{Remark}

\theoremstyle{plain}
\newcounter{maint}

\theoremstyle{plain}

\begin{document}

\renewcommand{\baselinestretch}{1.2}

\thispagestyle{empty}

\title[Techniques for classifying Hopf algebras]
{Techniques for classifying Hopf algebras and applications to dimension $p^3$}
\author[M. Beattie and  G. A. Garc\'ia]
{Margaret Beattie \and Gast\'on Andr\'es Garc\'\i a}
\thanks{M. Beattie's research was partially supported by an NSERC Discovery Grant.
G. A. Garc\'ia's research was partially supported by
 ANPCyT-Foncyt, CONICET, Ministerio de Ciencia y Tecnolog\'\i a (C\'ordoba)
 and Secyt (UNC)}
\address{\newline \noindent Department of Mathematics and Computer Science
\newline \noindent  Mount Allison University
\newline  \noindent Sackville, NB E4L 1E6
\newline \noindent Canada
\vspace*{0.5cm}
\newline \noindent Facultad de Matem\'atica, Astronom\'\i a y F\'\i sica \&
\newline \noindent Facultad de Ciencias Exactas, F\'\i sicas y Naturales,
\newline \noindent Universidad Nacional de C\'ordoba. CIEM -- CONICET.
\newline \noindent Medina Allende s/n
\newline \noindent (5000) Ciudad Universitaria, C\'ordoba, Argentina}
\email{ mbeattie@mta.ca} \email{ggarcia@famaf.unc.edu.ar}

\date{\today}

\begin{abstract}
The classification of all Hopf algebras of a given finite dimension
over an algebraically closed field of characteristic $0$ is a
difficult problem. If the dimension is a prime, then the Hopf
algebra is a group algebra. If the dimension is the square of a
prime then the Hopf algebra is a group algebra or a Taft Hopf
algebra.  The classification is also complete for dimension $2p$ or
$2p^2$,  $p$ a prime. Partial results for some other cases are
available. For example, for dimension $p^3$ the classification of the semisimple
Hopf algebras was done by Masuoka, and the pointed Hopf algebras were classified  by
Andruskiewitsch and Schneider, Caenepeel and D\u{a}sc\u{a}lescu, and
 \c{S}tefan and van Oystaeyen  independently. Many classification
results for the nonsemisimple, nonpointed, non-copointed
 case have been proved by the second author but the classification
in general for dimension $p^3$ is still
 incomplete, up to now even for dimension $27$.
\par In this paper we outline some results and techniques which have
been useful in approaching this   problem and add a few new ones.
 We give some further
results on Hopf algebras of dimension $p^{3}$
and finish the
classification for dimension $27$.
\end{abstract}

\maketitle

\section{Introduction}
The question of classification of all Hopf algebras of a given dimension
up to isomorphism
dates back over $35$ years  to Kaplansky's first monograph on
bialgebras \cite{Kap}. Some progress has been
made on this problem but, in general, it is a difficult question
where there are no standard methods of
attack.

\par  Hopf algebras of some dimensions over an algebraically
closed field $\Bbbk$ of characteristic zero
are completely classified.
By the Kac-Zhu Theorem \cite{Z}, 1994, the only Hopf algebras
of dimension $p$ prime are the group algebras $\Bbbk C_p$.
Hopf algebras of dimension $p^2$ are either pointed, and then
they are Taft Hopf algebras, or   they are semisimple and
are group algebras \cite{Ng}, \cite{masuoka-p^n}, \cite{AS p squared}.
\par
For $p=2$, the Hopf algebras of dimension $p^3 =8$ were classified by
Williams \cite{W}; different proofs
appear in \cite{ma-6-8} and \cite{stefan}. There are $14$ isomorphism
classes of Hopf algebras
of dimension $8$.
First there are three group algebras for the abelian groups of order $8$,
then there are the two group algebras for the
nonabelian groups of order $8$ and their duals which are nonpointed.
Finally there is a self-dual semisimple Hopf algebra of dimension $8$
which is neither commutative nor cocommutative.
This Hopf algebra can be constructed
 as an extension of $\Bbbk [C_2 \times C_2]$ by $\Bbbk C_2$.
Thus every Hopf algebra of dimension $8$ is either semisimple
 or pointed nonsemisimple or the dual is pointed nonsemisimple.

\par  For dimension $p^3$, $p$ odd,
 the classification question is still  open although the semisimple
and the pointed Hopf algebras of dimension
  $p^3$ have been classified.  For $H$ semisimple, the classification
is due to Masuoka \cite{ma-pp}.
There are $p+8$ isomorphism classes of semisimple Hopf algebras
of dimension $p^3$, namely the three group algebras
of abelian groups, the two group algebras of the nonabelian groups
and their duals and $p+1$ self-dual
semisimple Hopf algebras
which are neither commutative nor cocommutative.
These Hopf algebras are extensions of $\Bbbk [C_p \times C_p]$
by $\Bbbk C_p$. For $H$ pointed
nonsemisimple of dimension $p^3$, the classification is due
to Andruskiewitsch and Schneider \cite{AS2},
Caenepeel and D\u{a}sc\u{a}lescu \cite{CD}, and
 \c{S}tefan and van Oystaeyen \cite{Stv} independently.
There are $ (p-1)(p+9)/2$ isomorphism classes
and two of these have nonpointed duals.

 \par In \cite{GG}, the second author has proved some
classification results for Hopf algebras of dimension $p^3$ and
 conjectures that Hopf algebras of this dimension are either semisimple,
pointed nonsemisimple or the dual is pointed nonsemisimple.
It is shown in \cite{GG}
 that every ribbon Hopf algebra of dimension $p^3$ is either
a group algebra or a Frobenius-Lusztig kernel. As well,
 if $H$ is nonsemisimple, nonpointed   and the dual is not
pointed, then $H$
has no proper normal sub-Hopf algebra and is either
 of type $(p,1)$, meaning that the grouplikes of $H$ have order
$p$ and those of $H^\ast$ have order $1$,
 or else is of type $(p,p)$.  However, even for dimension $27$,
only the quasi-triangular Hopf algebras have been classified.

 \par In this note, we first develop some more tools for attacking
these classification problems.  Some of these are
 based on work of D. Fukuda \cite{fukuda-18}, \cite{fukuda-pq}
and use the coradical filtration.  Others extend
 a key theorem of \c{S}tefan \cite{stefan} on automorphisms
of matrix coalgebras.  We apply these results in the
 last section of this paper, but find them interesting and useful in their own right.

In the last section, we present some general
results about Hopf algebras of dimension $p^3$ and
 then in the final subsection we complete the classification
for dimension $27$ and show that all Hopf algebras of dimension
 $27$ are either semisimple, pointed nonsemisimple or the
dual is pointed nonsemisimple.  This result adds to the list of small dimensions
for which the classification is known.
The smallest open dimension   is
  $24$. Then the next is     $32$ since   the
classification for $27$ is done in this note, $29,31$ are prime, $25$ is a prime squared, $26$ is twice a prime,
$28$ was completed in \cite{ChNg}, and $30$ was completed in \cite{fukuda-30}.

\section{Preliminaries}
\subsection{Conventions}\label{subsec:conv}
Throughout, we will work over $\Bbbk$, an algebraically closed
field of characteristic zero.
 Since we are interested in classification,
our Hopf algebras will be understood to be finite dimensional unless otherwise
stated. Good references for Hopf algebra theory are \cite{Mo}, \cite{S}. 
Throughout $C_n$ will denote the cyclic group of order $n$ and
  $p$ will denote
an odd prime.

\smallbreak It is useful to recall that for $H$ a finite dimensional
Hopf algebra, if $B$ is a sub-bialgebra of $H$ then
$B$ is a sub-Hopf algebra.
To see this, let $\Phi:{\rm  End}(B) \rightarrow {\rm End}(B)$
be given by $\Phi(f) = f \ast \id_{B}$. Since $f(b_1)b_2 = g(b_1)b_2$
implies $f(b)=g(b)$ by applying
$\varepsilon_H$, we have that $\Phi$ is one to one,
 and since $B$ is finite dimensional
then $\Phi$ is onto.  Thus there exists $s \in {\rm End}(B) $
such that $\Phi(s) = u \varepsilon$,
\textit{i.e.} $s$ is a left inverse to the identity.
A similar argument yields a right inverse to
the identity  and these must be equal.
Then $s$ and $S_H$ are inverses to the identity in
${\rm Hom}(B,H)$ and so must be equal.

If $H$ is any Hopf algebra over $\Bbbk$ then $\com$, $\e$, $S$
denote respectively the comultiplication, the counit and the
antipode.  Comultiplication and coactions are written using
the Sweedler-Heynemann notation with summation sign suppressed.
The coradical $H_0$ of $H$ is the coradical of $H$ as a coalgebra, \textit{i.e.}
the sum
of all simple subcoalgebras of $H$.
$H$ is called cosemisimple if $H = H_0$.
With our assumptions on $\Bbbk$, by \cite{LR1, LR2}
a Hopf algebra $H$ is semisimple if and only
 if it is cosemisimple if and only if $S^2$ is the identity.
  $H$ is called pointed if $H$ has only simple subcoalgebras
of dimension $1$ and $H$ is called copointed if its dual is pointed.
We adopt the convention
that a pointed Hopf algebra means non-cosemisimple pointed,
that is, not a group algebra.

For $D$ a coalgebra, $G(D)$ denotes the grouplike elements of $D$ and
$(D_n)_{n \in \N}$ denotes the coradical filtration of $D$.
Let $L$ be a coalgebra with a distinguished grouplike 1. If $M$
is a right $L$-comodule via $\delta$, then the space of {\it right
coinvariants} is $$ M^{\co \delta} = \{x\in
M\mid\delta(x)=x\ot1\}.
$$ In particular, if $\pi:H\rightarrow L$ is a morphism of Hopf
algebras, then $H$ is a right $L$-comodule via $(\id\ot\pi)\com$ and
in this case $H^{\co \pi}:=H^{\co\ (\id\ot\pi)\com}$. Left coinvariants, written
${^{\co \pi} H }$ are defined analogously.

 $L_h$
(resp. $R_h$) is the left (resp. right) multiplication in $H$ by
$h$. The left and right adjoint action $\adl, \ad_r:H\rightarrow
\End(H)$, of $H$ on itself are given, in Sweedler notation, by:
$$
\adl(h)(x) =  h_1xS(h_2), \qquad \ad_r(h)(x) =
S(h_1)xh_2,
$$
for all $h,x\in H$. The set of $(h,g)${\it -primitives }(with
$h,g\in\GH$) and {\it skew-primitives} are:
$$
\begin{array}{rcl}
\cP_{h,g}(H)&:=&\{x\in H\mid\com(x)= g\ot x + x\ot h\},\\
\noalign{\smallskip} \cP(H)&:=&\sum_{h,g\in\GH}\cP_{h,g}(H).
\end{array}
$$
We say that $x\in \Bbbk(h-g)$ is a {\it trivial} skew-primitive;
otherwise, it is {\it nontrivial}.

\smallbreak Let $N$ be a natural number and let
$q$ be a primitive $ N $-th  root of unit. We denote by $ T_{q} $
the Taft Hopf algebra which is generated as an algebra by  elements $
g $ and $ x $ satisfying the relations $ x^{N} = 0 = 1-g^{N}$, $
gx=q xg $. It is a Hopf algebra of dimension $ N^{2} $ where the
comultiplication is determined by $ \com(g) = g\ot g $ and $ \com(x)
= x\ot 1 + g\ot x $.

\subsubsection{Extensions of Hopf algebras}\label{subsec:extensions}
We recall the definition of an exact sequence of Hopf algebras.

\begin{definition}\cite{andrudevoto}.
Let $A\overset{\imath}\hookrightarrow
H\overset{\pi}\twoheadrightarrow B$ be a sequence of Hopf algebra
morphisms. We  say that the sequence is {\it exact} if the following conditions hold:
\begin{enumerate}
\item[$ (i) $] $\imath$ is injective (and then we identify $A$ with its image);
\item[$ (ii) $] $\pi$ is surjective;
\item[$ (iii) $] $\pi\imath=\e$;
\item[$ (iv) $] $\ker\pi=A^+H$ ($A^+$ is the kernel of the counit);
\item[$ (v) $] $A=\,H^{\co \pi}$.
\end{enumerate}
\end{definition}

In such a case we will also say that $H$ is an {\it
extension of $A$ by $B$}. An exact sequence is called {\it central}
if $A$ is contained in the center of $H$.

The next theorem  is due to S. Natale but it is derived from
a result in \cite{stefan}. (See Theorem \ref{thm:stefan} and
Proposition \ref{prop:gen-stefan-vay}.)
Natale's theorem has been a key component
in classification results.

\begin{prop}\label{prop:natale-stefan}\cite[Prop. 1.3]{natale}. Let $H$ be a
finite dimensional nonsemisimple Hopf algebra. Suppose that $H$
is generated by a simple subcoalgebra of dimension $4$ which is
stable by the antipode. Then $H$ fits into a central exact
sequence $\Bbbk^G\overset{\imath}\hookrightarrow
H\overset{\pi}\twoheadrightarrow A,$ where $G$ is a finite group
and $A^*$ is a pointed nonsemisimple Hopf algebra.\qed
\end{prop}

The following statement condenses some known results and is useful
in finding exact sequences; for more details see \cite[Lemma 2.3]{GV}.

\begin{lema}\label{lem:exact-seq-dim}
Let $H$ be a finite dimensional Hopf algebra. If $\pi:H\rightarrow
B$ is an epimorphism of Hopf algebras then $\dim H=\dim H^{\co
\pi}\dim B$. Moreover, if $A=H^{\co \pi}$ is a sub-Hopf algebra
then the sequence $ A\overset{\imath}\hookrightarrow
H\overset{\pi}\twoheadrightarrow B $ is exact. \qed
\end{lema}

\subsection{Matrix-like coalgebras}

Let $\coMn$ denote the simple coalgebra of
dimension $n^2$, dual to the matrix algebra ${\mathcal M}(n,\Bbbk)$.
We say that a coalgebra $C$ is a $d\times d$ \textit{matrix-like
coalgebra} if $C$ is spanned by elements $(e_{ij})_{1\leq i,j\leq n}$
called a \textit{matrix-like spanning set} (not necessarily linearly
independent) such that $\com(e_{ij}) =
\sum_{1\leq \ell \leq n} e_{i\ell} \ot e_{\ell j}$ and $\e(e_{ij}) = \delta_{ij}$.
If the matrix-like spanning set $(e_{ij})_{1\leq i,j\leq d}$ is linearly independent,
following \c Stefan we call $\be =  \{e_{ij}:\ 1\leq i,j\leq d\}$
a \textit{multiplicative matrix}
and then $ C\cong \coMd $ as coalgebras.

If $\pi: \M^\ast(n,\Bbbk) \rightarrow D$ is a coalgebra map,
then there are various possibilities for
the image.  From \cite[Thm. 2.1]{bitidasca}
we have the following theorem that describes them when $n=2$.

\begin{theorem}\label{thm:2x2-matrix}
Let $\pi: \M^\ast(2,\Bbbk) \rightarrow D$ be
a coalgebra map and denote by $C$ the image of $\M^\ast(2,\Bbbk)$
 in $D$.
If $\dim C = 3$, then   $C$ has a basis $\{g,h,u \}$ with
 $g,h$ grouplike and $u$ an $(h,g)$-primitive, and
  if $\dim C =1, 2$, then $C $ has a basis consisting of grouplike elements.  \qed
\end{theorem}

The next theorem due to  \c{S}tefan has turned out
to be crucial for several classification results including the proof of
Proposition \ref{prop:natale-stefan}.
Similar arguments
were used before by Larson and Radford in \cite[Section 6]{LR3} to study the
order of the antipode in a semisimple Hopf algebra.

\begin{theorem}\label{thm:stefan}\cite[Thm. 1.4]{stefan} Let $D = \coM$.
\begin{enumerate}
 \item[$(i)$] For $f$  an antiautomorphism of $D$ such
that $\ord ( f^{2}) = n < \infty$ and $n > 1$, there exists a multiplicative
matrix $\be$ for
$D$ and a root of unity $\omega$ of order $n$ such that
$$f(e_{12}) = \omega^{- 1} e_{12},\quad f(e_{21}) = \omega e_{21},\quad
f(e_{11}) = e_{22},\quad  f(e_{22}) = e_{11}.$$

\item[$(ii)$] For $f$  an automorphism of $D$ of finite order $n$,
 there exist a multiplicative matrix $\be$ for $D$
and a root of unity $\omega$ of order $n$ such that $f(e_{ij}) = \omega^{i-j} e_{ij}$.
\qed
\end{enumerate}
 \end{theorem}

The proof of the following generalization
is due to C. Vay.
Although similar results seem to be
well-known, we include it for completeness.

\begin{prop} \label{prop:gen-stefan-vay}  Let $D = \coMd$.
\begin{enumerate}
\item[$(i)$] For $f$  an automorphism of $D$ of finite order $n$,
 there exist a multiplicative matrix $\be$ for $D$
and $\omega_1, \ldots , \omega_ d\in\Bbbk^\times$ such
that $f(e_{ij})=\omega_i\omega_j^{-1}e_{ij}$ and
$\omega_i\omega_j^{-1}$ are $n$-roots of unity.
\item[$(ii)$] Let $f$ be an anti-automorphism of $D$
with $1<n=\ord(f^2)<\infty$.
Then there exists a multiplicative matrix $\be$ for $D$,
a symmetric matrix $A_+\in\GL(\ku,a_+)$, an
anti-symmetric matrix $A_-\in\GL(\Bbbk,a_-)$, $A_j\in\GL(\Bbbk,a_j)$
and $\lambda_j\in\Bbbk^\times$ for
$j=1,\ldots , s$ such that $d=a_++a_-+2a_1+\cdots+2a_s$
and the matrix of $f$ corresponding to $\be$ is
\begin{align*}
A&=\left[
   \begin{matrix}
   0 & 0 & 0 & 0 & 0 & 0 & 0 & \lambda_1A_{1}^t \\ 
   0 & 0 & 0 & 0 & 0 & 0 & \diagup & 0 \\ 
   0 & 0 & 0 & 0 & 0 & \lambda_{s} A_{s}^t & 0 & 0 \\ 
   0 & 0 & 0 & A_{+} & 0 & 0 & 0 & 0 \\ 
   0 & 0 & 0 & 0 & A_{-} & 0 & 0 & 0 \\ 
   0 & 0 & A_{s} & 0 & 0 & 0 & 0 & 0 \\ 
   0 & \diagup & 0 & 0 & 0 & 0 & 0 & 0 \\ 
   A_{1} & 0 & 0 & 0 & 0 & 0 & 0 & 0
  \end{matrix}
  \right].
\end{align*}
Moreover,
\begin{enumerate}
\item[$(a)$] $\lambda_ i^2$, $\lambda_i\lambda_j^{-1}$,
$\lambda_i\lambda_j$ are $n$-roots of unity.
\item[$(b)$] If $a_+\neq0$ then $\lambda_i$ are $n$-roots of unity.
\item[$(c)$] If $a_-\neq0$ then $-\lambda_i$ are $n$-roots of unity.
\end{enumerate}
\end{enumerate}
In particular, for $d=3$ and $n>2$
\begin{align*}
&&A_\lambda=\left[
   \begin{matrix}
   0 & 0 & \lambda \\ 
   0 & 1 & 0\\ 
   1 & 0 & 0 \\ 
  \end{matrix}
  \right],
&& f(\be)=
\left[
   \begin{matrix}
   e_{33} & \lambda e_{23} & \lambda e_{13} \\ 
   \lambda^{-1} e_{32} & e_{22} & e_{12}\\ 
   \lambda^{-1} e_{31} & e_{21} & e_{11} \\ 
  \end{matrix}
  \right] && \text{ and }\qquad f^{2}(e_{ij}) = \lambda^{j-i}e_{ij}.
\end{align*}
\end{prop}

\begin{proof}
Let $f$ be a coalgebra automorphism as in $(i)$.
Fix a multiplicative matrix $\tilde\be =
(\tilde{e}_{ij})_{1\leq i,j\leq n}$ for $D$ and
let $\tilde{B}$ be the matrix of $f$ corresponding to
$\tilde\be $.
Then $\tilde{B}^n=b\id$ for some $b\in\Bbbk^\times$.
In particular, there exist $U\in\GL(\Bbbk, d)$ and
$\omega_1, \ldots , \omega_d\in\Bbbk^\times$
such that $U\tilde{B}U^{-1}=\diag(\omega_1, \ldots, \omega_d)$.
If we take $\be=U\tilde{\be}U^{-1}$ then $f$ is afforded by
$\diag(\omega_1, \ldots, \omega_d)$ and $(i)$ follows.

Now suppose that $f$ is a coalgebra anti-automorphism as in
$(ii)$
and $\tilde{A}$ be the matrix of $f$ corresponding to $\tilde\be$.
Then $f^2$ is afforded by $B=\tilde{A}(\tilde{A}^{-1})^t$.
As before, there exist $U\in\GL(\Bbbk, d)$ and
$\omega_1, \ldots , \omega_d\in\Bbbk^\times$
such that $UBU^{-1}=\diag(\omega_1, \ldots, \omega_d)$.
Since $n>1$, $\omega_{i_0}\neq1$ for some $i_0$.
If we take $\be=U\tilde{\be}U^{-1}$ then $f$ is afforded by
$A=U\tilde{A}U^{t}$ with respect to $\be$ and
$A(A^{-1})^t=UBU^{-1}=\diag(\omega_1, \ldots, \omega_d)$
is the matrix of $f^2$  corresponding to $\be$. Therefore,
$a_{ij}=\omega_ia_{ji}$ for all $1\leq i,j\leq d$ and hence
$a_{ij}= \omega_{i}\omega_{j}a_{ij}$ for all $1\leq i,j\leq d$.
 If $w_{j}w_{i_0}\neq1$ for all $1\leq j\leq d$, we have that
 $a_{ji_0}=a_{i_0j}=a_{i_0i_0}=0$. Since $A\in\GL(\Bbbk, d)$ it is impossible.
 Then for all $i$ there exists $j$ such that $\omega_i\omega_j=1$.
 If $d=2$ then $A=\left[\begin{smallmatrix}0&a_{12}\\ a_{21}&0
 \end{smallmatrix}\right]$ and Theorem
 \ref{thm:stefan} follows.
 We perform a new change of basis using a
 matrix permutation such that $\{\omega_{1}, \ldots, \omega_{d}\}$
$=\{\lambda_1, \ldots, \lambda_1, \ldots, \lambda_s, \ldots, \lambda_s, 1,
\ldots, -1,\lambda_s^{-1},\ldots, \lambda_s^{-1},
\ldots, \lambda_1^{-1},\ldots, \lambda_1^{-1}\}$.
If $I_{\pm}=\{i\leq d|\omega_i=\pm1\}$ and $I_{\pm l}=
\{i\leq d|\omega_i=\lambda_l^{\pm1}\}$ then $a_{ij}=0$
for all $(i,j)\notin \bigcup_{l=1}^sI_{\pm l}\times
I_{\mp l}\cup I_{+}^2\cup I_{-}^2$
since $a_{ij}=\omega_i\omega_ja_{ij}$. Then
\begin{align*}
A&=\left[
   \begin{matrix}
   0 & 0 & 0 & 0 & 0 & 0 & 0 & \lambda_1A_{1}^t \\ 
   0 & 0 & 0 & 0 & 0 & 0 & \diagup & 0 \\ 
   0 & 0 & 0 & 0 & 0 & \lambda_{s} A_{s}^t & 0 & 0 \\ 
   0 & 0 & 0 & A_{+ } & 0 & 0 & 0 & 0 \\ 
   0 & 0 & 0 & 0 & A_{-} & 0 & 0 & 0 \\ 
   0 & 0 & A_{s} & 0 & 0 & 0 & 0 & 0 \\ 
   0 & \diagup & 0 & 0 & 0 & 0 & 0 & 0 \\ 
   A_{1} & 0 & 0 & 0 & 0 & 0 & 0 & 0
  \end{matrix}
  \right],
\end{align*}
where $A_{\pm}$ is a $|I_{\pm}|\times|I_{\pm}|$-matrix
and $A_{l}$ is a $|I_{-l}|\times|I_{+l}|$-matrix.
Since $A$ is invertible, we must have that $|I_{-l}|=|I_{+l}|$.
Moreover, by the relations above, $A_+$ is symmetric
and $A_-$ is a anti-symmetric matrix.
Then set $a_{\pm}:=|I_{\pm}|$ and $a_l:=|I_{+l}|$
for all $l=1, \ldots, s$. The items $(a)$, $(b)$ and $(b)$
follow from $(i)$ for $f^2$.

\par If $d=3$ and $n>2$, the unique possibility
up to a change of basis is
$\diag\{\omega_1,\omega_2,\omega_3\} =
\diag\{\lambda, 1, \lambda^{-1}\}$. If $ f $ is an
anti-automorphism, we have that
$a_+=1$, $a_-=0$, $ a_{1}=1 $. In such case,
we make a new change of basis using
$\diag(A_1^{-\frac{1}{2}},A_+^{-\frac{1}{2}}, A_1^{-\frac{1}{2}})$
and the rest of the claim follows by direct computation.
\end{proof}


\section{The coradical filtration and some classification results}
\label{sec: coradical filtration}

In this section we present some results on dimensions of Hopf
algebras using a description of the coradical filtration due
to Nichols and presented  in \cite{andrunatale}, and more
recent results by D. Fukuda.  First we recall the definitions.
More detail can be found in \cite[Section
1]{andrunatale}.

Let $D$ be a coalgebra over $\Bbbk$ and denote by $(D_{n})_{n\in \N}$
its coradical filtration.
Then there exists a coalgebra projection
$\pi:D \to D_{0}$ from
$D$ to the coradical
$D_{0}$ with kernel $I$, see
\cite[5.4.2]{Mo}. Define the maps

$$\rho_{L}:= (\pi\ot \id)\com: D \to D_{0}\ot D \qquad\mbox{ and }\qquad
\rho_{R}:= (\id\ot \pi)\com: D \to D\ot D_{0},$$

\noindent and let $P_{n}$ be the sequence of subspaces defined recursively
by
\begin{align*}
P_{0} & = 0,\\
P_{1} & = \{x\in D:\ \com(x) = \rho_L(x) +\rho_R(x)\}
= \com^{-1}(D_{0}\ot I + I\ot D_{0}),\\
P_{n} & = \{x\in D:\ \com(x) - \rho_L(x) - \rho_R(x) \in
\sum_{1\leq i \leq n-1}P_{i}\ot P_{n-i}\}, \quad n\geq 2.
\end{align*}

Then by a result of Nichols, $P_{n} = D_{n}\cap I$ for $n\geq 0$,
see \cite[Lemma 1.1]{andrunatale}. Suppose that
$D_{0} = \bigoplus_{\tau \in \II} D_{\tau}$, where the $D_{\tau}$ are simple
coalgebras and $\dim D_{\tau} = d^{2}_{\tau}$.
Any $D_{0}$-bicomodule is a direct sum of simple $D_{0}$-sub-bicomodules
and every simple $D_{0}$-bicomodule has coefficient coalgebras
$D_{\tau}, D_{\gamma}$ and has dimension
$d_{\tau}d_{\gamma} = \sqrt{\dim D_{\tau}\dim D_{\gamma}}$
for some $\tau, \gamma \in \II$, where $d_{\tau},d_{\gamma}$
are the dimensions of the associated comodules of $D_{\tau}$ and
$D_{\gamma}$, respectively.

Let $H$ be a   Hopf algebra. Then $H_{n},P_{n}$ are
$H_{0}$-sub-bicomodules of $H$ via $\rho_R$ and $\rho_L$. As in
\cite{andrunatale}, \cite{fukuda-pq}, for all $n\geq 1$ we denote by
$P_{n}^{\tau,\gamma}$ the isotypic component of the
$H_{0}$-bicomodule of $P_{n}$ of type the simple bicomodule with
coalgebra of coefficients $D_{\tau}\ot D_{\gamma}$. If $D_{\tau} =
\Bbbk g$ for $g$ a grouplike, we may use the superscript $g$ instead
of $\tau$.  Let $P^{\tau,\gamma} = \sum_{n \geq 0}P_{n}^{\tau,\gamma}$.

For $\Gamma$ a set of grouplikes of a Hopf algebra $H$,
let
  $ P^{\Gamma,\Gamma}$ denote  $\sum_{g,h \in \Gamma}P^{g,h}
 $
  and let $H^{\Gamma, \Gamma}:= P^{\Gamma, \Gamma}
\oplus \Bbbk\Gamma$. If $\mathcal{D},\mathcal{E}$ are
   sets of simple subcoalgebras, let $P^{\mathcal{D}, \mathcal{E}}$ denote
   $\sum_{D \in \mathcal{D}, E \in \mathcal{E}}P^{D,E}$.

Let $H_{0,d}$ with $d\geq  1$ denote the
direct sum of the simple subcoalgebras of $H$ of dimension $d^{2}$.
It was noted in \cite{andrunatale} that $H_n = H_0 \oplus P_n$ and
   $|G(H)|$ divides the dimensions of $H_n$, $P_n$
and $H_{0,d}$ for all $n\geq 0$ and
   all $d \geq 1$, by the Nichols-Zoeller theorem.

Following D. Fukuda, we
say that the subspace $P_{n}^{\tau,\gamma}$ is
\textit{nondegenerate} if $P_{n}^{\tau,\gamma} \nsubseteq P_{n-1}$.
The next  results are due to D. Fukuda; note that $(ii)$  is a
generalization of \cite[Cor. 1.3]{andrunatale} for $n>1$.

\begin{lema}\label{lema:fukuda-deg} \label{lema:fukuda}
\label{lema:fukuda-deg-m}
Let
$D_{\tau}, D_{\gamma}$ be simple subcoalgebras of a Hopf algebra $H$.
\begin{enumerate}
 \item[$(i)$] \cite[Lemma 3.2]{fukuda-pq} If the subspace $P_{n}^{\tau,\gamma}$
is nondegenerate for some $n > 1$, then there exists a set of
simple coalgebras $
\{D_{1},\cdots ,D_{n-1} \}$ such that
$P_{i}^{\tau,D_{i}}$, $P_{n-i}^{D_{i},\gamma}$ are nondegenerate
for all $1\leq i\leq n$.
 \item[$(ii)$] \cite[Lemma 3.5]{fukuda-pq}
 For $g \in \GH$,
$\dim P_{n}^{\tau, \gamma} = \dim P_{n}^{S\gamma,S\tau}
= \dim P_{n}^{g\tau, g\gamma}= \dim P_{n}^{\tau g, \gamma g}$,
 where the superscript $S\alpha$ means that the simple coalgebra is
$S( D_{\alpha})$ and the superscript
$g\alpha$ or $\alpha g$ means
that the coalgebra is $g D_{\alpha}$ or
$D_{\alpha}g$.
 \item[$(iii)$] \cite[Lemma 3.8]{fukuda-pq} Let $C,D$ be
simple subcoalgebras such that $P_{m}^{C,D}$ is nondegenerate. If
$\dim C \neq \dim D$ or $\dim P_{m}^{C,D} - P_{m-1}^{C,D} \neq \dim
C$ then there exists a simple subcoalgebra $E$ such that
$P_{\ell}^{C,E}$ is nondegenerate for some $\ell\geq m+1$. \qed
\end{enumerate}
\end{lema}

The following proposition generalizes  \cite[Cor. 4.3]{bitidasca}
by giving a better lower bound for the dimension
of a non-cosemisimple Hopf algebra with
no nontrivial skew-primitive elements.

\begin{prop}\label{prop:biti-dasca}\label{cor:bitidasca-p1}
Let $H$ be a non-cosemisimple Hopf
algebra with no nontrivial skew-primitives.
\begin{enumerate}
  \item[$(i)$]
For any $g \in G(H)$ there exists a simple subcoalgebra $C$ of
$H$ of dimension $> 1$ such that $P_{1}^{C,g}\neq 0$, $P_k^{C,D}$ is
nondegenerate for some $k>1$ and $D$ simple of the same dimension as
$C$, and $P_m^{g,h}$ is nondegenerate for some $m>1$ and $h$
grouplike.
\item[$(ii)$]
Suppose $H_0\cong \Bbbk G\mas{\mathcal
M}^*(n ,\Bbbk) \oplus \sum_{i=1}^t{\mathcal M}^*(n_i,\Bbbk)$ with $t\geq 0$,
$2\leq n \leq n_1 \leq \ldots \leq n_t$. Then
$$
\dim H \geq \dim(H_0) + (2n + 1)|G| + n^2.
$$
\end{enumerate}
\end{prop}

\begin{proof}  Part $(i)$ follows from \cite[Cor. 4.2]{bitidasca} and
Lemma \ref{lema:fukuda-deg}$(iii)$. Part $(ii)$ follows from \cite[Cor.
4.3]{bitidasca}, Lemma \ref{lema:fukuda-deg-m} $(ii)$ and part $(i)$.
\end{proof}

We end this section with  a series of results whose proofs use the
results  of D. Fukuda in Lemma \ref{lema:fukuda}.

\begin{prop}\label{prop:1plusC}
Let $H$ be a Hopf algebra  with coradical $H_0 = \Bbbk \cdot1 \oplus
E$ where $E$ is the direct sum of simple coalgebras, each of
dimension divisible by $N^2$ for some $N > 1$. Let $\dim H \equiv d
\mod N$ with $0 \leq d \leq N-1$.  Let $1 \leq e \leq N$ with $e
\equiv d-1 \mod N$  and assume that $e \neq 1$. Then   $ \dim H
\geq \dim H_0 + 4N + 2N^2 + e$.
\end{prop}

\begin{proof} Let $\E$ denote the set of simple subcoalgebras
of dimension greater than $1$.
  Throughout, let $X,Y,Z,W$ denote  elements of $\E$.  Then $N$ divides the
dimension of $E, P^{\E,1}, P^{1,\E}, P^{\E,\E}$
 so that by our assumptions and the fact that $P^{1,1} \neq 0$ we have
$\dim P^{1,1} \equiv e \mod N$;  in particular $\dim
P^{1,1}
> 1
$.

Since by Proposition \ref{prop:biti-dasca}$(i)$ there are $X,Y$
such that  $P_1^{X,1}$,  $P_1^{1,Y}  $ are nonzero, we get that
\begin{align*}
 M  & = { \rm max}\{  m   |\ P_m^{X,1}
\mbox{ is nondegenerate for
some   }X \text  \}\\
& = {\rm max} \{ m |\ P_m^{1,Y} \mbox{ is
nondegenerate for some   }Y   \} \geq 1.
\end{align*}
Then by Lemma \ref{lema:fukuda-deg}$(iii)$,
  there is an integer $k >M$   such that
$P_k^{1,1}$ is nondegenerate. By Lemma \ref{lema:fukuda-deg}$(i)$,
$P_k^{1,1}$ nondegenerate   and the fact that $P_1^{1,1} = 0$,
implies that for some   $X \in \E$, $P_1^{1,X}$ and $P_{k-1}^{X,1}$ are
nondegenerate. Thus $k-1 \leq M <k$ and so $k = M+1$ and $M+1$ is
the largest integer $m$ such that $P_m^{1,1}$ is nondegenerate.

\par  Suppose that $M=1$ so $M+1 =2$. Then $P_2^{1,1} = P^{1,1}$ and
the dimension of $P^{1,1}_2$ must be greater than 1.
 Thus $\dim P_2^{1,1} -
\dim P_1^{1,1}
> 1$ and, by Lemma \ref{lema:fukuda-deg}$(iii)$, $P_t^{1,1}$
is nondegenerate for some $t >2$,
a contradiction. Thus $M \geq 2$ and $M+1 \geq 3$.

If $P_k^{1,X}$ is nondegenerate then the difference in
dimensions between $P_k^{1,X}$ and $P_{k-1}^{1,X}$ is a multiple of
$N$. Thus, the fact that we have nondegenerate spaces
$P_M^{1,X},P_M^{Y,1},P_1^{1,Z},P_1^{W,1}$ with $X,Y,Z,W$ not
necessarily distinct yields that the sum of the dimensions of the
$P^{1,\E}$ and $P^{\E,1}$ is at least $4N$.

By Lemma \ref{lema:fukuda-deg}$(iii)$, $P_\ell^{Y,Y^\prime}$ is
nondegenerate for some $ Y^\prime \in \E$ and $\ell > M$. By Lemma
\ref{lema:fukuda-deg}$(i)$, if $P_M^{1,X}$ is nondegenerate, then
$P_{M-1}^{Z,X}$ is nondegenerate for some $Z\in \E$.  Thus $\dim
P^{\E,\E} \geq 2N^2$.

Hence  $\dim H = \dim H_0 +    \dim P^{1,1}  + 2 \dim
P^{\E,1}   +   \dim P^{\E,\E} \geq \dim H_0 + e + 4N + 2N^2$.
\end{proof}

\begin{example}\label{ex: dim}
  Suppose that $H$ has dimension $N^3$ with $N>2$.
Then $H$ cannot have
coradical $\Bbbk \cdot 1 \oplus  \mathcal{M}^\ast(N, \Bbbk)^t$
 where $t = N-1$ or
$t = N-2$. For then we would have $N^3 \geq 1 + tN^2 + N - 1  + 4N +
2N^2
  \geq  (N-2)N^2 +     2N^2 + 5N
 \geq N^3 + 5N$, which is impossible.

 In particular, if $\dim H= 27$, then $H$
cannot have coradical $k\cdot 1
\oplus \mathcal{M}^\ast(3,\Bbbk)^t$ with $t = 1,2$
and if $\dim H=125$, $H$ cannot have coradical $\Bbbk \cdot 1 \oplus
\mathcal{M}^\ast(5,\Bbbk)^t$ with $t = 3,4$.
\end{example}

The next proposition applies the methods of this section to say
something about Hopf algebras
 of dimension $p^3$, $p$ an odd prime,
leading into the material in the next section.

\begin{proposition}\label{prop: dim simples dim 4}
Let $H$ be a non-copointed Hopf algebra of dimension $p^3$
with $p$ an odd prime.
Then $H$ has no simple $4$-dimensional
 subcoalgebra stable under the antipode $S$ so that
$H_0 \ncong \Bbbk \cdot 1 \oplus \M^\ast(2, \Bbbk)$.  Furthermore, if
  $H_0 \cong \Bbbk\cdot 1 \oplus \mathcal{M}^\ast(2, \Bbbk)^t$
with $t >1$, then
  \begin{displaymath}
\dim H \geq \dim H_0 + 24 \text
{ if } p \equiv 1 \mod 4 \quad\text{ and }\quad
\dim H \geq \dim H_0 + 22
\text{ if } p \equiv 3\mod 4.
\end{displaymath}
\end{proposition}

\begin{proof}
Let $\D$ be the set of simple subcoalgebras
of dimension $4$ and suppose that $D \in \D$ is
stable under the antipode $S$.
 Then $D$ must generate all of $H$
since any proper sub-Hopf algebra of $H$ is a
group algebra or a Taft Hopf algebra.
 Thus by Proposition \ref{prop:natale-stefan}, $H$ has pointed dual,
contradicting our assumptions. Thus for each
$D \in \D$, $S(D) \neq D$. This also shows that
$H_0 \ncong \Bbbk \cdot 1 \oplus
\M^\ast(2, \Bbbk)$.

\par
Now suppose that  $H_0 \cong
\Bbbk\cdot 1 \oplus \mathcal{M}^\ast(2, \Bbbk)^t$ with $t >1$.
The argument is similar to that in the previous proposition. Since
$2 \dim P^{\D, 1}$ and $\dim P^{\D,\D}$ are divisible by $4$, then
  $p^3 -1   \equiv \dim P^{1,1} \mod 4$.

Since $p \equiv p^3\mod 4$, then $P^{1,1}$ has dimension
a nonzero multiple of $4$
   if $p \equiv 1\mod 4$ and dimension a nonzero multiple of
$2$ if $p \equiv 3\mod 4$.  As in the proof of Proposition
\ref{prop:1plusC}, let $M$ be the largest integer such that
$P_M^{1,D}$ is nondegenerate for $D \in \D$ and then $\ell = M+1$ is the
largest integer such that $P_\ell^{1,1}$ is nondegenerate.
If $M+1= \ell=2$ then $P^{1,1}_2 = P^{1,1}$ has dimension $2$ or $4$
contradicting Lemma \ref{lema:fukuda}$(iii)$.  Thus $M \geq 2$ and $\ell
\geq 3$.
 Suppose that $\ell=3$ and $M=2$.

By  Lemma \ref{lema:fukuda}, there exist
$D,E \in \D$ so that $P_1^{1,D}, P_2^{D,1},P_1^{D,E},P_1^{E,1}$ are
nondegenerate. Thus twice the dimension of $P^{\D,1}$ is at least
$8$.

Consider the dimension of $P^{\D,\D}$.
If $D=E$ in the paragraph above,
then for some $X,Y \in \D$, and some $k>1$, the spaces
$P_1^{D,D}, P_1^{S(D),S(D)}, P_3^{D,X}, P_k^{S(D), Y}$ are nondegenerate
so the dimension of $P^{\D,\D}$ is at least $16$.
If $E = S(D)$, then there exists $X \in \D$ and $k >1$, such that
$P_1^{D,S(D)}, P_k^{S(D),X}, P_3^{D,X}$
is a set of  nondegenerate spaces.
Then the dimension of $P^{\D,\D}$ is at least $12$.
Finally, if $D\neq E \neq S(D)$, then the spaces
$P_1^{D,E}, P_1^{S(E),S(D)}, P_3^{D,X}, P_k^{E,Y}$
for some $k>1$ and $X,Y \in \D$
are nondegenerate. Again in
this case the dimension of $P^{\D,\D}$ is at least $16$.

\par Thus the dimension of $H$ is at least
$\dim H_0 + \dim P^{\D,\D} + 2 \dim P^{1, \D} +
\dim P^{1,1} \geq \dim H_0 + 20
 + \dim P^{1,1}$ and the statement is proved.
\end{proof}

\begin{example} \label{example: dim 27 no simples dim 4}
Let $H$ be a non-copointed Hopf algebra of dimension $27$,
then $H_0 \ncong \Bbbk \cdot 1 \oplus \M^\ast(2,\Bbbk)^t$
for any $t \geq 1$.
\end{example}

The last example in this section  uses similar
techniques to prove the existence of nontrivial
skew-primitive elements in a Hopf algebra
with nontrivial grouplikes.

\begin{example} Let $H$ be a nonsemisimple nonpointed
and non-copointed Hopf algebra of dimension $5^3$.
Then by \cite{GG} (see Proposition \ref{prop: GG prop 3.15}),
$H$ has no simple subcoalgebra of dimension $4$ fixed by
the antipode.
Moreover, if $H_0 \cong \Bbbk C_5
\oplus \mathcal{M}^\ast(2, \Bbbk)^5$, then
$H$ must have a nontrivial skew-primitive element.

To see this, assume $H$ has no nontrivial skew-primitive element.
By Proposition \ref{prop:biti-dasca}$(i)$,
$P_1^{1,D} \neq 0$ for some
simple subcoalgebra $D$ of dimension $4$.
 Since $S^{4p}$ is the identity by Radford's formula for $S^4$
and $S(D) \neq D$, we must have
that the simple subcoalgebras of dimension $4$ are $D_0:= D$,
$D_1 = S(D), \ldots, D_{4}= S^{4}(D)$.
 Thus $P_1^{1,D}, P_1^{D_1,1},
P_1^{1,D_2}, \ldots, P_1^{1,D_{4}},P_1^{D,1}, \ldots, P_1^{1,D_{4}}$
are all nonzero and of equal dimension.
 By Lemma \ref{lema:fukuda}$(ii)$ and the fact that $G(H) = C_5$,
  $\dim P^{C_5,\mathcal{D}} + \dim P^{\mathcal{D},C_5} \geq
2\cdot 25 \cdot \dim P^{1,D} \geq 100 $ where $\mathcal{D}$
denotes the set of simple $4$-dimensional coalgebras.
By Lemma \ref{lema:fukuda}$(iii)$, $P^{D_i,\mathcal{D}}$ is
nonzero for every $i$ so that  the dimension
of $P^{\mathcal{D},\mathcal{D}}$ is at least $20$. Also
$\dim P^{C_5,C_5} $ must be divisible by $4$ and also by $5$ so
is at least $20$.
 Thus  $  \dim H_0 +
 2\dim P^{C_5,\mathcal{D}} +
\dim P^{\mathcal{D},\mathcal{D}}+
\dim P^{C_5,C_5} \geq 165 $, a contradiction.\end{example}

\section{Hopf algebras of dimension $p^{3}$}
Throughout this section, $H$ will be a Hopf algebra of dimension $p^{3}$.
Since semisimple and nonsemisimple pointed Hopf algebras have
already been classified, we will assume that $H$ is nonsemisimple,
nonpointed and non-copointed.
We denote by $T_{q}$ the Taft Hopf algebra of dimension
$p^{2}$ for $q$ some primitive
$p$-th root of unity. For $y$ a $(1,g)$-primitive in $T_q$
we will write $T_q = \Bbbk \langle g,y \rangle$;
here $gy = qyg$.
We first summarize some results from \cite{GG}.

 \subsection{Known results for dimension $p^3$}
The general classification of Hopf algebras of
dimension $p^3$ was studied by the second
author in \cite{GG} where it was conjectured that a
Hopf algebra of dimension $p^3$ is semisimple
or pointed or copointed.  Since these Hopf algebras
have all been classified, a complete list of isomorphism
types can be given if the conjecture holds. In particular,
it is shown in \cite{GG} that ribbon Hopf algebras
of dimension $p^3$ are group algebras or Frobenius-Lusztig kernels.

\par Also by \cite{GG}, we know that if a Hopf algebra $H$ of
dimension $p^3$ is nonsemisimple, nonpointed
and non-copointed, then $H$ is simple as a Hopf algebra,
meaning that it has no normal proper sub-Hopf algebras. Furthermore
the only possible types are $(p,p)$ and $(p,1)$ so that
 $H$  has grouplikes of order $p$, and the dual
has grouplikes of order $p$ or $1$.  In
particular, $S^{4p} = \id$ by Radford's formula.

\begin{theorem}\cite[Thm. 2.1]{GG}\label{thm: 2.1 GG}
Let $L$ be a finite dimensional Hopf algebra which
fits into an extension $
  A  \overset{i}{\hookrightarrow}  L \overset{\pi}{\twoheadrightarrow} \Bbbk C_p
$, where $A^\ast$ is pointed and $|G(A^\ast)| \leq p$.  Then $L$ is copointed.
\qed
\end{theorem}

\begin{proposition}\cite[Cor. 3.14]{GG}\label{prop: GG cor 3.14}
Let $L$ be a Hopf algebra of dimension
$p^3$ and type $(p,p)$ such that $L$ contains a
nontrivial skew-primitive element.  Then $L$
is a bosonization of $\Bbbk C_p$.\qed
\end{proposition}

\begin{proposition}\cite[Prop. 3.15]{GG}
\label{prop: GG prop 3.15} Let $L$ be a nonsemisimple Hopf algebra of
dimension $p^3$ and assume that $L$ contains a simple
subcoalgebra of dimension $4$ fixed by the antipode. Then $L$ is
copointed and furthermore, $L$ cannot be of type $(p,p)$.
\qed
\end{proposition}

In \cite{GG} the complete classification of the quasi-triangular
Hopf algebras of dimension $27$ was given.
We are able to complete the classification for dimension $27$ in the
general case in
Subsection \ref{subsect: 27}.

\subsection{Some general results for dimension $p^3$}
In this subsection we prove some general results for Hopf algebras
of dimension $p^3$ and grouplikes of
 order $p$, especially those containing a   sub-Hopf algebra isomorphic to $T_{q^\prime}$
for some $q^\prime$
  and/or having a Hopf algebra projection to some
 $T_{q}$, with
$q$ not necessarily equal to $q^\prime$,
 that is, the
Hopf algebra contains a nontrivial skew-primitive element
and/or its dual does.
Since by Proposition \ref{prop:biti-dasca}$(ii)$
 every non-cosemisimple Hopf algebra of dimension $27$
with grouplikes of order $3$ has a nontrivial skew-primitive element, these
 results will lead to the classification for dimension $27$.

\smallbreak The next proposition is interesting itself but
also has a number of useful consequences.

\begin{proposition}\label{prop: pointed}  Let $L$ be a Hopf algebra of dimension $p^3$.
Assume $L$ contains
a   sub-Hopf algebra
$T \cong T_q$ and suppose
that for some $n\geq 1$
and grouplike $g$, $L$ contains an element $y \in P_n^{g,1}$
with $y \notin P_{n-1}^{g,1}$   and
$y \notin T$ such that $\Delta(y) = g \otimes y + y \otimes 1 +
y^\prime$ where $y^\prime \in T \otimes T$. Then $L$ is pointed.
\end{proposition}

\begin{proof} Since $T \varsubsetneq L$,  $p\leq |G(L)| < p^{3} $.
If $|G(L)| = p^{2}$, then $L$ is pointed by \cite[Prop. 3.3]{GG}. Assume
then that
$G(L) = G(T) \cong C_{p}$.
Let $A_0 := T$ and
$A_1$ be the vector space spanned by $ T  \cup \{ ayb: a,b \in T  \}$.
Then $A_0A_1 \subseteq A_1$ and
 $A_1 A_0 \subseteq A_1$.
 Let $K$ be the algebra
generated by $A_1$, in other words by $T$ and $y$.
Since the comultiplication is an algebra map, $K$ is a sub-bialgebra of $L$,
and thus a sub-Hopf
algebra. Since $p^2 < \text{dim}(K) \leq p^3$, we have that $K = L$.
Furthermore, $\Delta(A_1) \subseteq
A_0 \otimes A_1 + A_1 \otimes A_0$ so that by \cite[5.5.1]{Mo},
it follows that $A_0 \supseteq L_0$ and so $L$ is pointed.
\end{proof}

Recall that throughout this section $H$ denotes a
nonsemisimple nonpointed non-copointed Hopf algebra
of dimension $p^3$.

\begin{cor}\label{cor: s-p dim p}
Assume $H$ contains a pointed
sub-Hopf algebra $T\cong T_{q}$. Then
all skew-primitive elements of $H$ lie in $T$. In particular, the dimension
of the space of nontrivial skew-primitive elements is $p$. \qed
\end{cor}

  \begin{prop}\label{pr: divides p-1}
  Suppose that $H$ has grouplikes of order $p$
and $P_1^{g,E} = 0
  = P_1^{E,g}$ for all $g \in G(H)$ and all
simple subcoalgebras $E$ of dimension greater
  than $1$. Then
$P^{g,E} = 0
  = P^{E,g}$ for all $g \in G(H)$ and all simple
subcoalgebras $E$ of dimension greater
  than $1$.

If we assume further
that $H$ contains a pointed sub-Hopf algebra $T\cong T_{q}$, then
$ H^{ C_p,  C_p} = T$ and
if every simple subcoalgebra of dimension
greater than $1$ of $H$ has
  dimension divisible by $d^2 < p^2$, then $d^2$ must divide $p-1$.
  \end{prop}

\begin{proof}
Suppose that $P_k^{g,E} = 0$ for all $k<n$ and
$P_n^{g,E} $ is nonzero.   Then
by Lemma \ref{lema:fukuda}$(i)$, there are simple
subcoalgebras $D_i$, $1 \leq i \leq n-1$ such that
$P_i^{g,D_i}$ and $P_{n-i}^{D_i,E}$ are nondegenerate.
But the first condition implies that the $D_i$ all
have dimension $1$ and the second condition then gives a
contradiction to the induction assumption.

Since clearly $T \subseteq H^{C_p,C_p}$,
it remains to show that $P^{g,h} \subset T$ for all
$g,h \in G(H)$. By Corollary \ref{cor: s-p dim p}, $P_1^{g,h} \subset T$.
Suppose that $P_k^{g,h} \subset T$ for
all $k <n$ and let $0 \neq y \in P_n^{g,h}$ with $y \notin P_{n-1}^{g,h}$.
  Then $\Delta(y) = g \otimes y + y \otimes h + y^\prime$
where $y^\prime \in \sum_{1 \leq i \leq n-1}P_i \otimes P_{n-i}$.
Since, by $(i)$, for
 $g$ any grouplike,  $P_i^{g,D} = 0 = P_{n-i}^{D,g}$ for $1 \leq i < n$
and for $D$ simple of dimension greater than $1$,
then $y^\prime \in P_i^{C_p,C_p} \otimes P_{n-i}^{C_p,C_p} \subset T \otimes T$
by the induction assumption.
By Proposition \ref{prop: pointed}, $y \in T$, which implies that
$P^{g,h} \subset T$ and hence $ H^{ C_p,  C_p} = T$.

Suppose now that $H_0 = \Bbbk C_p \oplus E$
where $E$ is the sum of  simple subcoalgebras of
dimension divisible by $d^2 >1$. Let $\E$
denote this set of simple subcoalgebras.  Then
   as a coalgebra $H \cong H^{ C_p,  C_p} \oplus E \oplus P^{\E,\E}
=T \oplus E \oplus P^{\E,\E} $.
    Thus $\dim H - \dim T = p^3 -p^2 = p^2(p-1)$
must be divisible by $d^2$ and
  since $d<p$, $d^2$ must
   divide $p-1$.
\end{proof}

\begin{remark}\label{rem: m divides p-1}
 Suppose that $G(H) \cong C_p$ and  $H$ contains a pointed
sub-Hopf algebra $T\cong T_q$
 such that   $T = H^{C_p,C_p}$.  If for some $d>1$,
any simple subcoalgebra   of dimension greater than $1$ has
 dimension divisible by $d^2$, then $d$ must divide $p-1$ and if $d$ is even,
then $2d$ divides $p-1$. The argument is the same as
 that above, except that here possibly   $P^{C_p,\E}$ is nonzero.
\end{remark}

\begin{cor}\label{cor: H0 not kCp + p(p-1)sq}
 Suppose that $G(H) \cong C_p$.  Then $\Ho \ncong \Bbbk C_{p}
\oplus \mathcal{M}^\ast(p-1,\Bbbk)^{sp} \oplus E$
as coalgebras where  $E = 0$ or $E = \M^\ast(p,\Bbbk)^t$
with $s, t \geq 1$ .
\end{cor}

\begin{proof} Suppose first that $E = 0$ and
$\Ho \cong \Bbbk C_{p}
\oplus \mathcal{M}^\ast(p-1,\Bbbk)^{sp}$.  By
Proposition \ref{prop:biti-dasca}$(ii)$, if $H$ has no nontrivial
skew-primitives then
$ \dim H \geq p + sp(p-1)^2 + 2p(p-1) + (p-1)^2 + p = p^3 + (p-1)^2 +p
> p^3$, a contradiction. Thus,
 $H$  must contain a pointed
sub-Hopf algebra of dimension $p^2$.

Next suppose that $P_1^{g,D  } \neq 0 $ for some
$g \in G(H)$, $D \in \D$, where $\D$ is the set of  simple coalgebras of
dimension $(p-1)^2$. Then
   by Lemma
\ref{lema:fukuda}$(ii)$, $ \dim P^{C_p, \D}
 = \dim P^{\D, C_p} \geq p(p-1)$.
By
Lemma \ref{lema:fukuda}$(iii)$,
  $\dim P^{\D,\D}$
 is a nonzero multiple of  $(p-1)^2$. In fact, since $p$
divides  $\dim H_0$, $\dim P^{C_p,C_p}$ and
$\dim P^{C_p, \D}$, then $\dim P^{\D,\D}$ is a nonzero multiple of $p(p-1)^2$.
 Thus \begin{displaymath} \dim H \geq p^2  + sp(p-1)^2
+2p(p-1) + p(p-1)^2   \geq 2p^3 -p^2 > p^3,
\end{displaymath}
a contradiction.  Hence $P^{C_p,\D} = 0$ and by Proposition \ref{pr: divides p-1},
$(p-1)^2 $ must divide $p-1$, an impossibility.  Thus $\Ho \ncong \Bbbk C_{p}
\oplus \mathcal{M}^\ast(p-1,\Bbbk)^{sp}$.
\par Now suppose that $E  \cong \mathcal{M}^\ast(p,\Bbbk)^t$.
 The usual argument
shows that $H$ has  a sub-Hopf algebra
isomorphic to $T_q$.  But then $p^3
\geq p^2 + sp(p-1)^2 + tp^2 \geq p^3 + (t-1)p^2+p$, a contradiction.
 \end{proof}

\begin{remark}\label{rem: has sps}
 Suppose that $H$ has coradical  $  \Bbbk C_{p}
 \oplus \mathcal{M}^\ast(p ,\Bbbk)^n$.
Then by Proposition \ref{prop:biti-dasca}$(ii)$,
 if $H$ has no nontrivial skew-primitive, the dimension of $H$ is at least
 $(n+3)p^2 + 2p $. Thus if $(n+3)p^2 + 2p > p^3$,
then $H$ contains a pointed sub-Hopf
 algebra isomorphic to $T_q$. In particular, if $H_0 \cong  \Bbbk C_{p}
 \oplus \mathcal{M}^\ast(p ,\Bbbk)^n$ with $n \geq p-3$, then $H$
has a skew-primitive element.
\end{remark}

\begin{cor} \label{cor: H0 not kCp + p(p-2)sq}
Assume $p\geq 7$, then for $s \geq 1$,
  $\Ho \ncong \Bbbk C_{p}
\oplus \mathcal{M}^\ast(p-2,\Bbbk)^{sp}  $   as coalgebras.

\end{cor}
\begin{proof} We assume $s=1$; the proof for $s > 1$ is the same.
 First we show that $H$ must
have a skew-primitive element. Suppose not,
then by Proposition \ref{prop:biti-dasca}$(i)$,
$P^{C_p, \E} \neq 0$ and $P^{\E,\E} \neq 0$
  where $\E$ is the set of simple subcoalgebras
 of dimension $p-2$. By the usual dimension
arguments $\dim P^{\E,\E}$ is divisible by $p$ so that
 $ \dim H_0 + 2 \dim P^{C_p, \E} + \dim P^{\E,\E} +
\dim P^{C_p,C_p} = p + p(p-2)^2 + 2(p)(p-2) + p(p-2)^2
 = 2p^3 + -6p^2 + 5p = p^3 + [p^6(p-6) + 5p] > p^3$,
which is impossible. Thus $H$ has a sub-Hopf algebra
$T \cong T_q$ for some $q$.

Now a similar argument shows that $P^{C_p,\E} = 0$.
If not, then the dimension of $H$ is
 $p^2 + 2p(p-2)^2 + 2p(p-2) = p^3 + [p^2(p-5) + 4p] > p^3$ ,
a contradiction.
Thus we may apply Proposition \ref{pr: divides p-1} to obtain the result.
\end{proof}

\begin{cor}
Let $p \geq 19$ be such that $p-1$ is not divisible by $4$. Then
$H_0 \ncong \Bbbk C_p \oplus
\mathcal{M}^\ast(p-5,\Bbbk)^p \oplus \mathcal{M}^\ast(p-3, \Bbbk)^p$.
\end{cor}

\begin{proof}
By Proposition \ref{prop:biti-dasca}$(ii)$, if $H$ has no
skew-primitives, then \begin{displaymath} p^3
\geq p + p(p-5)^2 + p(p-3)^2 + (2p-10 + 1)p + (p-5)^2
= 2p^3 -13p^2 +16p + 25
\end{displaymath}
which is impossible if $p \geq 19$. Thus $T_q \varsubsetneq H$ for some $q$.
  If $P^{g,E}_1$ for some
$E$ simple of dimension at least $(p-5)^2$, then
the same counting argument as in
Corollary \ref{cor: H0 not kCp + p(p-1)sq}
gives a contradiction. The statement then follows
from Proposition \ref{pr: divides p-1}.
\end{proof}

The next propositions give some information about the
coradical of  Hopf algebras of type $(p,p)$
such that the dual contains a Taft Hopf algebra.

\begin{prop}\label{prop:p3-type-p-p}
Let $H$ be of type $(p,p)$ with $G(H) = \langle c \rangle$
and suppose there is a Hopf
algebra projection
$\pi: H \twoheadrightarrow T_q$.
Then $\pi(c) \neq 1$.
If, also, $H$  contains a sub-Hopf algebra
$T_{q^\prime} = \Bbbk\langle c,x \rangle $,
then $\pi(x) = 0$
and $H^{co\pi}$ has basis $\{ 1, x, \ldots, x^{p-1} \}$.
\end{prop}

\pf
If $\pi(c) = 1$, then $\Bbbk \langle c \rangle \cong \Bbbk C_{p}
\subseteq H^{\co \pi}$. Since
$\dim H^{\co \pi} = p$, by Lemma \ref{lem:exact-seq-dim}
the sequence of Hopf algebras
$ \Bbbk C_{p} \hookrightarrow H \twoheadrightarrow T_{q} $
is exact, which implies by
Theorem \ref{thm: 2.1 GG} that $H$ is pointed.

If $H$ has a skew-primitive element $x$ and
  $\pi(x) \neq 0$, then necessarily $T_{q} \cong T_{q'}$ and hence $H$
is a bosonization of $T_{q}$. This implies that it is
pointed by \cite[Thm. 8.8]{AS2}.
\epf

\begin{prop}\label{prop: corad not p copies of D}
Let $H$ be of type $(p,p)$ with $G(H) = \langle c \rangle $
and suppose there is a Hopf
algebra projection
$\pi: H \twoheadrightarrow T_q$.
Assume further that $H_0 \cong \Bbbk \langle c \rangle
\oplus  \mathcal{M}^\ast(2, \Bbbk)^{tp} \oplus E$ where
  $t\geq 1$ and   $E$
is either $0$ or the sum of
simple coalgebras of dimension strictly greater than $4$.   Then
\begin{enumerate}
 \item[$(i)$] The simple subcoalgebras of dimension $4$
are nontrivially permuted by $L_c$ ($R_c$),  left (right)
multiplication by $c$ and also by the antipode.
 \item[$(ii)$] The number of simple subcoalgebras of dimension
$4$ must be strictly greater than $p$.
 \item[$(iii)$] If $t=2$ then the simple subcoalgebras of dimension $4$
are stable under $S^4_H$.
\item[$(iv)$] If some simple subcoalgebra is fixed by $S^4_H$,
then there is a simple subcoalgebra $D$ of dimension $4$
with multiplicative matrix $\bf{e}$ such that $e_{22},
e_{21} $ are in ${^{co\pi}H}$ and $c^{-1}e_{11}, c^{-1}e_{21}$
are in $H^{co\pi}$.
\end{enumerate}
\end{prop}

\begin{proof}$(i)$
For any simple subcoalgebra $D$ of $H$, the algebra generated
 by $D$ is a bialgebra, thus a sub-Hopf algebra, and hence all of $H$.
Therefore $\pi(D)$, the image
of $D$ in $T_q$ must generate $T_q$ as an algebra.
Suppose $T_q = \Bbbk\langle  g,y \rangle$ with
$\Delta(y) = g \otimes y + y \otimes 1$.

If $\dim D = 4$, by Theorem \ref{thm:2x2-matrix},
$\pi(D)$ must have dimension $3$ and have basis $\{
g^iy, g^i, g^{i+1} \}$ for some $i \geq 0$.
Denote this subcoalgebra of $T_q$ by $\langle g^iy \rangle$.
By Proposition \ref{prop:p3-type-p-p},
we may assume that $\pi(c) = g$ and thus $\pi(c^jD)
= \pi(Dc^j) =  \langle g^{i+j} y \rangle$.
Thus $L_c$ and $R_c$ nontrivially permute the simple subcoalgebras of
dimension $4$. Also  $S_{T_q}(y) = -g^{-1}y$, so that
 $S_{T_q}(\langle g^iy \rangle ) = \langle g^{-1}yg^{-i}
\rangle = \langle g^{-i-1}y \rangle$ and
$S$ also properly permutes the simple subcoalgebras of dimension $4$.

\par $(ii)$ Suppose that $t=1$ so that for each $i$,
there is exactly one $4$-dimensional simple
subcoalgebra which maps by $\pi$ to $\langle g^iy \rangle$.
Since $\langle g^i y \rangle$ is fixed by $S_{T_q}^2$ then the
simple subcoalgebras of dimension $4$ are fixed by $S_H^2$
but permuted nontrivially by $S_H$.
 Since there is an odd number of  simple subcoalgebras
of dimension $4$, then one of these must be fixed by
$S_H$, a contradiction
by Proposition \ref{prop: GG prop 3.15}.

\par $(iii)$ If $t=2$, then for each $i$ there are exactly two
$4$-dimensional simple subcoalgebras
which map by $\pi$ to any $\langle g^iy \rangle$.
If these are not stable under $S^2_H$, then they are
exchanged by $S^2_H$ and thus stable by $S^4_H$.

\par $(iv)$ Suppose $D$ is a simple $4$-dimensional
subcoalgebra stable under $S_H^4$.  By $(i)$, we may
assume  that   $\pi(D) = \langle y \rangle$.
By Radford's formula for $S^4$,
since $H$ and $H^\ast$ have grouplikes of order
$p$, $S^{4p}_{H}=\id_{H}$. Also
 $S_{T_{q}}$ has order $2p$.
Thus $ \ord S_H^4=p$ and also $\ord S_{T_q}^4=p$.
Since $gy = qyg$, $S^4(y) = q^{-2}y$.
Thus the subcoalgebra
$\langle y \rangle$ of $T_q$ is the
direct sum of the eigenspace for $1$ with basis $\{1,g \}$
and the eigenspace for $q^{-2}$ with basis $y$.
Now apply Theorem \ref{thm:stefan}
to obtain a multiplicative matrix $\bf{e}$ for $D$ such
that $e_{11}, e_{22}$ are a basis for the eigenspace for $1$
for $S_H^4$ and $e_{12}$ is an eigenvector for $q^{-2}$.
Then $e_{21}$ is an eigenvector for $q^{2}$ and thus
$\pi(e_{21}) = 0$. Then $\pi(e_{ii})$ is grouplike in
$\langle y \rangle$ and $\pi(e_{12})$ is $(\pi(e_{22}), \pi(e_{11}))$-
primitive.  Thus $\pi(e_{22}) = 1, \pi(e_{11}) = g, \pi(e_{12}) =
\alpha y$ where $\alpha$ is a scalar, and by rescaling if necessary
we can assume that $\pi(e_{12}) = y$.  It is now immediate that $e_{22},
e_{21} $ are in ${^{co\pi}H}$ and $c^{-1}e_{11}, c^{-1}e_{21}$
are in $H^{co\pi}$.
\end{proof}

In the next proposition, we assume that both
$H$ and $H^\ast$ have nontrivial skew-primitive elements and
we can then show that $H$ cannot contain a
$4$-dimensional simple subcoalgebra.

\begin{prop}\label{prop: no simple 4 dim}
Suppose $H$ contains a   sub-Hopf algebra
 $T \cong T_{q^\prime}  $ and suppose
also that there is a Hopf algebra projection
$\pi: H \rightarrow T_q = \Bbbk \langle g,y \rangle$.
Then $H$ cannot contain a simple subcoalgebra of dimension $4$.
\end{prop}
\begin{proof} Suppose $D \cong \M^\ast(2, \Bbbk)$ is a subcoalgebra of $H$.
 Then,
 as in the proof of Proposition \ref{prop: corad not p copies of D},
$\pi(D)$ generates $T_q$
so that   $\pi(D)= \langle g^iy \rangle $
has basis $\{ g^i, g^iy, g^{i+1} \}$ for some
$0 \leq i \leq p-1$.

\par Then $\ker \pi_{|_D} = \Bbbk z$ for some $z \in D$
and is a coideal in $D$.  Since $\pi(z) = 0$, then
$\varepsilon(z) = 0 $ too.  Thus
$\Delta(z) = a \otimes z + z \otimes b$
where $a,b \in H$ with $\varepsilon(a) = \varepsilon(b) = 1$.
Then $(\Delta \otimes \id_D) \circ \Delta (z)
= (\id_D \otimes \Delta) \circ \Delta (z)$ implies that
\begin{displaymath}
\Delta(a) \otimes z + \Delta(z) \otimes b   =
a \otimes \Delta(z) + z \otimes \Delta(b).
\end{displaymath}
Applying $\id_D \otimes \pi \otimes \pi$
and $\pi \otimes \pi \otimes \id_D$ to the above, we obtain that
\begin{displaymath}
z \otimes \pi(b) \otimes \pi(b) =
z \otimes (\pi \otimes \pi) \Delta (b);
\hspace{2mm} (\pi \otimes \pi) \Delta(a)
 \otimes z = \pi(a) \otimes \pi(a) \otimes z.
\end{displaymath}
Thus $\pi(a), \pi(b)$ are grouplike in $\pi(D)$.
If $\pi(b) = g^j$, since by Proposition \ref{prop:p3-type-p-p}
there is some $c \in G(H)$ with $\pi(c) = g^{-j}$, then
\begin{displaymath}
(\id_D \otimes \pi) \circ \Delta ( c z) = cz \otimes \pi(c b) = cz \otimes 1,
\end{displaymath}
and $cz \in H^{co \pi}$, which
is a contradiction by Proposition \ref{prop:p3-type-p-p}.
\end{proof}

%

In the next example  we give
another proof of the
 impossibility of coradical $\Bbbk C_3 \oplus
\mathcal{M}^\ast(2,\Bbbk)^3$ for $H$
 of dimension $27$
   since the argument uses the dimension
of projective covers of simple $H$-comodules
 and illustrates  different techniques for these problems.

\begin{example}\label{lem:3-4to3}
Assume $\dim H =27$ and $H$ is of type $(3,3)$. Then
$\Ho \ncong \Bbbk C_{3} \oplus \coM^{3}$  as coalgebras.

To see this, suppose $\Ho \cong \Bbbk C_{3} \oplus \coM^{3}$
as coalgebras.
Then the isomorphism classes of
the simple  $ H^{*} $-modules are the
 three one-dimensional modules  $\Bbbk_{\eps}, \Bbbk_{g},
\Bbbk_{g^{2}}$ and three two-dimensional modules $ V_{1},
V_{2} $ and $ V_{3} $. If we denote by $ P(V) $ the projective cover
of a simple module $ V $, then we have that
$ 27 = 3 \dim P(\Bbbk_{\eps}) + \sum_{i=1}^{3}2\dim P(V_{i}).
$
Since by \cite[Proof of Prop. 2.1]{etinofgelaki2},
$\dim P(\Bbbk_{g^{j}}) \geq \dim P(\Bbbk_{\eps}) $ for $j=1,2$,
and $P(\Bbbk_{\eps})\otimes \Bbbk_{g^{j}}$ is projective,
it follows that $P(\Bbbk_{\eps})\otimes \Bbbk_{g^{j}}
\cong P(\Bbbk_{g^{j}})$ for $j=1,2$. Analogously,
$P(V_{1})\otimes \Bbbk_{g^{j-1}}
\cong P(V_{j})$ for $j=1,2$. Indeed,
as the simple modules $V_{j}$ are not stable
by right tensoring with $ \Bbbk_{g} $, since otherwise
$3$ would divide $\dim V_{i} = 2$ by \cite[Prop. 2.5]{etinofgelaki2},
we may assume that $V_{j} \cong V_{1}\ot \Bbbk_{g^{j-1}}$
for $j=1,2$. Then there is a surjection from the projective module
$P(V_{1})\otimes \Bbbk_{g^{j-1}}$ to
$ V_{1}\ot \Bbbk_{g^{j-1}} \cong V_{j}$ which
implies that it also has a projection to
$ P(V_{j}) $ and in particular, $ \dim P(V_{1}) \geq
\dim P(V_{j}) $. Since the action given by right tensoring
with $ \Bbbk_{g} $ is transitive, we have that
$  \dim P(V_{j}) \geq
\dim P(V_{1}) $   for $j=1,2$ which implies that
$P(V_{1})\otimes \Bbbk_{g^{j-1}}
\cong P(V_{j})$ for $j=1,2$. Consequently we have
that $ 27 = 3\dim P(\Bbbk_{\eps}) + 6\dim P(V_{1})$.

\par We claim now that $ \ord S = 6 $. Indeed,
by Radford's formula for the antipode
we know that $ \ord S \vert 12$. Assume that $ \ord S \neq 6 $ and
let $ D $ be a simple $4$-dimensional subcoalgebra of $ H $.
Then $ D $ generates $ H $ as an algebra since the subalgebra
generated by $ D $ is a sub-Hopf algebra
and there is no nonsemisimple
and nonpointed Hopf algebra of dimension $p$ or $ p^{2} $ by \cite{Z}
and \cite{Ng}.
Since the coradical of $ H $ contains $3$
simple subcoalgebras of dimension $ 4 $, we have that
$S^{6}(D) = D$ and by Theorem \ref{thm:stefan},
there exists a comatrix basis $ (e_{ij})_{1\leq i,j\leq 2} $
such that $ S^{6}(e_{ij}) = (-1)^{i-j}e_{ij} $. On the
other hand, since $ H^{*} $ contains a sub-Hopf
algebra isomorphic to a Taft Hopf algebra $ T_{q} $, there
exists a Hopf algebra surjection $\pi: H\to T_{q}$. But then,
$ (-1)^{i-j}\pi(e_{ij}) = \pi(S^{6}(e_{ij})) =
S^{6}(\pi(e_{ij}))= \pi(e_{ij})$, which implies
that $\pi(e_{ij}) = 0$ for $i\neq j$ and consequently
$\pi(D) \subseteq \Bbbk G(T_{q})$. This is impossible
since $D$ generates $ H $ as an algebra and therefore
$\pi(D)$ generates $ T_{q} $ as an algebra too.

As $ \ord S = 6 $, we have that $3$ divides
$\dim P(\Bbbk_{\eps})$ by \cite[Cor. 1.5]{ChNg}.
Since $ 27 = 3\dim P(\Bbbk_{\eps}) + 6\dim P(V_{1})$
and $ \dim P(V_{1}) \geq 2 $, the only possibility
is that $ \dim P(\Bbbk_{\eps}) =3 $ and $ \dim P(V_{1}) =3 $.
But this cannot occur, for if $ P(V)\neq V $ for a simple
module $ V $, then $ \dim P(V) \geq 2\dim V $.
Hence, there is no Hopf algebra of dimension
27 whose coradical contains $ 3 $ simple
subcoalgebras of dimension $ 4 $.\end{example}

Next we consider the number of simple subcoalgebras of dimension $9$
in a Hopf algebra $H$ of dimension
$p^3$ where both $H$ and $H^\ast$
have a nontrivial skew-primitive element.

\begin{prop}\label{prop: multiple of psquared}  Let $p>3$.
Assume $H$ has a pointed sub-Hopf algebra
$T \cong T_{q^\prime}$ and suppose
there is a Hopf algebra projection
$\pi: H \twoheadrightarrow T_q = \Bbbk \langle g,y \rangle$.  Then
the number of simple subcoalgebras of $H$ of
dimension $9$ is a multiple of $p^2$.
\end{prop}
\begin{proof}
If $H$ does not contain a simple subcoalgebra of dimension $9$ there
is nothing to prove. Assume then that $H$ contains a simple
subcoalgebra $D \cong \M^\ast(3, \Bbbk)$.
 As usual, by Proposition \ref{prop:p3-type-p-p}, there is a grouplike element
$c \in H$ such that $\pi(c) = g$.  We show that  the set $\D$ of
subcoalgebras $c^iDc^j$, $0 \leq i,j \leq p-1$, must have $p^2$
distinct elements. Since if $E$ is another simple subcoalgebra of
dimension $9$ such that $E \notin \D$, then for $\E$ the set of
coalgebras $c^iEc^j$, $\D \cap \E = \emptyset$, this will prove the
statement.
\par
 Suppose  that   two coalgebras in the  set
$\{ c^i D c^j | 0 \leq i,j \leq p-1\}$ are equal.
Thus $D = c^i D c^j$ for some $i,j$, not both $0$.
Let $\phi:
D \rightarrow D$ denote the coalgebra automorphism
of order $p$ defined by $\phi(d) = c^idc^j$
and then $\pi \circ \phi(d) = g^i \pi(d) g^{j}$.
First we show that $i+j \equiv 0\mod p$.

Suppose not. For the coalgebra automorphism  $  \phi^\prime$   of
  $T_q$  given
by $w \mapsto g^i w g^j$, let $T_q^{[\ell]}$
be the eigenspace for $q^\ell$, $0\leq \ell <p$.
Then a basis for $T_q^{[\ell]}$ is
$\{ (\sum_{t=0}^{p-1} q^{-t(kj + \ell)}g^{t(i+j)})y^k|\ 0 \leq k <p   \}$.
Since $\pi(D)$ generates $T_q$,
then $\pi(D) \nsubseteqq \Bbbk \langle g \rangle$ and thus some
element $(\sum_{t=0}^{p-1} q^{-t(kj + \ell)}g^{t(i+j)})y^k \in \pi(D)$
with $k >0$ lies in $\pi(D)$.
But then $g^m y^k $ and $g^m$ lie in $\pi(D)$ for all
$0 \leq  m \leq p-1$ so that the dimension
of $\pi(D)$ is at least $2p$. But since $p>3$, $2p>9$ so this is impossible and thus
$i+j \equiv 0\mod p$. Then
$\phi = \adl{c^i}$ and $D$ is also stable under $\adl{c}$.
Thus we assume that $i=1$.

Now let $A := (a_{ij})$ be the $3 \times 3$
matrix of eigenvalues for the eigenvectors $e_{ij}$ in the
multiplicative matrix $\bf{e}$ for $\adl{c}$ on $D$ so that
in the notation of Proposition \ref{prop:gen-stefan-vay},
$a_{ij} = \omega_i \omega_j^{-1}$.
 In $T_q$, a basis of eigenvectors for $\adl{g}$
for the eigenvalue $q^m$ is $\{g^ny^m | 0 \leq n \leq p-1  \}$.
Again, we denote this eigenspace by $T_q^{[m]}$.
If every entry in the first row of $A$ is $1$,
then $\omega_1 = \omega_2 = \omega_3$.
Since $T_q^{[0]} = \Bbbk \langle g \rangle$ and $\pi(D)$
generates $T_q$ we cannot have that every entry of $A$ is $1$
and so this is impossible.  Suppose that exactly two of the
entries in the first row of $A$ equal $1$, say $a_{11} = a_{12} = 1$.
Then $\omega_1 = \omega_2$ and so the first two rows of $A$
are equal and the first two columns of $A$ are equal.
Since $\pi(D)$ generates $T_q$ as an algebra, some $a_{ij}$
must be $q$. Suppose that $a_{13} = a_{23} = q$
and then $a_{31} = a_{32}=q^{-1}$.
(The argument will be the same if we swap $q$ and $q^{-1}$.)
\smallbreak
Next note that if $a_{ij}=q^n$ with $n \geq 3$ then $\pi(e_{ij}) =0$.
For if $\pi(e_{ij}) \neq 0$ then $\pi(e_{ij})
= \psi(g)y^n$ where $\psi(g) \in \Bbbk \langle g \rangle$ so that
$g^my^n \in \pi(D)$ for some $0\leq m<p$.
Since $\pi$ is a coalgebra map this implies that
$g^my^i \in \pi(D)$ for $0 \leq i \leq n
$ so that $\dim \pi(D) \geq n+1 \geq 4$.
Also $g^{m}y \in \pi(D)$ so that $g^{m+1}y^{n-1} \in \pi(D)$
and thus $g^{m+1}y^i \in \pi(D)$ for $0 \leq i \leq n-1$.
Similarly $g^{m+2}y^{n-2}, g^{m+2}y^{n-3}
\in \pi(D)$ so $\dim \pi(D) \geq 2n + 3 \geq 9$, a contradiction.

\par Now assuming that $a_{11} = a_{12} = 1$ and $a_{13} =
a_{23} = q$ implies that $a_{31} = a_{32} = q^{-1}$
so that $\pi(e_{31}) = \pi(e_{32}) = 0$. Then $\pi(e_{33})$
is grouplike and $\pi \otimes \id \circ \Delta(e_{33})
= \pi (e_{33}) \otimes e_{33}$.  Thus $c^m e_{33}
\in\ ^{co\pi}H$ for some $m$, contradicting Proposition \ref{prop:p3-type-p-p}.

\par Therefore exactly one entry in the first row of
$A$, namely $a_{11}$, equals $1$.  Thus $\pi(e_{11})$
is grouplike in $\pi(D)$ and $\pi(e_{12} ) \otimes \pi(e_{21}) + \pi(e_{13})
\otimes \pi(e_{31}) = 0$.
If  $\pi(e_{12}), \pi(e_{13})$ are both $0$ (both nonzero)
 then $e_{11} \in\ ^{co\pi}H$ ($H^{co\pi}$ respectively), and again
Proposition \ref{prop:p3-type-p-p} provides a contradiction.
Suppose $\pi(e_{12}) = 0 $ and $\pi(e_{13}) \neq 0$ so that
$\pi(e_{31}) = 0$. But $\pi(e_{12}) = 0 $ implies that
$\pi(e_{13}) \otimes \pi(e_{32}) = 0$ so that $\pi(e_{32}) = 0$.
Then $\pi(e_{22}) $ is grouplike and $e_{22} \in H^{co\pi}$, contradiction.

Thus the set $\D$ contains $p^2$ distinct
coalgebras of dimension $9$ and the proof is finished.
\end{proof}

\begin{cor}\label{cor: no dim 3 squared}
If $p=5$  or $7$ in
Proposition \ref{prop: multiple of psquared}
then $H$ has no simple subcoalgebras
of dimension $9$.
\end{cor}
\begin{proof}
Since $T \subset H$, if $H $ has a simple subcoalgebra
of dimension $9$, then the dimension of $H$ is
at least $p^2 + 9 p^2 = 10 p^2$ and this is impossible for $p<10$.
\end{proof}

\begin{cor}Suppose $H$ and $H^\ast$ are of type $(p,p)$
and each contains a skew-primitive element. Let $\D$ be the set of simple
subcoalgebras of dimension $9$ in $H$.
\begin{enumerate}
 \item[$(i)$] Then $P^{C_p, \D} = 0  $.
 \item[$(ii)$] If $p-1$ is not divisible by $9$ then
$H_0 \ncong \Bbbk C_p \oplus \M^\ast(3,\Bbbk)^{tp^2}$.
\end{enumerate}
\end{cor}

\begin{proof}$(i)$ By Proposition \ref{prop: multiple of psquared}
and Lemma \ref{lema:fukuda}$(ii)$,
 the dimension of $P^{C_p, \D}$ is at least $p^3$.

$(ii)$ Suppose $H_0 \cong \Bbbk C_p \oplus \M^\ast(3,\Bbbk)^{tp^2}$.
Then by Proposition \ref{pr: divides p-1}$(iii)$, $9$ must divide $p-1$.
\end{proof}

\smallbreak

Finally  consider the simple subcoalgebras of dimension
$p^2$ in a Hopf algebra of dimension $p^3$.

\begin{lema}\label{lemma: permutes}
Assume $H$ is of type $(p,p)$ with $G(H) = \langle c \rangle$  and suppose
there is a Hopf algebra projection
$\pi: H \rightarrow T_q$.
If $D$ is a simple subcoalgebra of $H$ of dimension $p^2$ then $L_c$,
left multiplication by $c$, nontrivially permutes
either the right or left simple subcomodules of $D$.
\end{lema}

\begin{proof} Let $T_{q} = \Bbbk \langle g,y \rangle$.
By Proposition \ref{prop:p3-type-p-p} we may assume
that $\pi (c) = g$.
Since $L_c$, left multiplication by $c$, is a coalgebra
bijection of order $p$ in $\mathrm{End}(H)$ and since there are
fewer than $p$ simple subcoalgebras of dimension $p^2$ in $H$,
then $L_c$ maps $D$ to $D$.   Similarly,
$\adl{c}$, the left adjoint action
of $c$ on $H$ is a coalgebra automorphism of $D$.
By Proposition \ref{prop:gen-stefan-vay}, there is a multiplicative matrix
$\mathbf{e}$ in $D$
for $\adl{c}$.
   Then
$ \pi \circ \adl{c} = \adl{g} \circ \pi$.  Since $gy = qyg$,
with $q$ a primitive $p$th root of unity,
 the eigenspace for the eigenvalue $q^i$ for
$\adl{g}$ in $T_q$ is $T_q^{[i]}:=
\{g^jy^i | 0 \leq j \leq p-1   \}$.  Using similar notation,
let $D^{[i]}$ denote the eigenspace for $q^i$ for
$\adl{c}$
in $D$.
The Hopf algebra projection $\pi$ maps  $D^{[i]}$  to  $ T_q^{[i]}$.

 Let $M_j$ denote the simple left subcomodule of $D$
with basis $\{e_{1j}, \ldots, e_{pj}   \}$ and let $R_j$ denote the
 simple right subcomodule of $D$ with basis $\{e_{j1}, \ldots, e_{jp}   \}$.
  Suppose that both the $R_i$ and the $M_i$ are stable
under $L_c$.  Then  for all $i,j$, $L_c(e_{ij}) = \alpha e_{ij}$ for some
  nonzero scalar $\alpha$; in other words, $\mathbf{e}$ is
a multiplicative matrix of eigenvectors both for
$\adl{c}$ and $L_c$.
    If $i=j$, by applying $\varepsilon$, we see that $\alpha = 1$
and since $e_{jj} \in D^{[0]}$ and
       $\pi(e_{jj}) = g\pi(e_{jj})$ then for all $j$,  $\pi(e_{jj}) =
t = (1/p)\sum_i g^i$, the integral in $\Bbbk \langle
g \rangle$.  Similarly for $i\neq j$, if $e_{ij} \in D^{[0]}$,
then  $\pi(e_{ij}) = \beta f_\lambda$ where
$f_\lambda:= [1 + \lambda g + \lambda^2 g^2 + \ldots +
\lambda^{p-1}g^{p-1}]$ for $\lambda$ a primitive $p$th root of $1$.
Note that $g f_\lambda = \lambda^{-1}f_\lambda$.
\smallbreak
Suppose that $R_1$ contains exactly $n\geq 1$
eigenvectors for the eigenvalue $1$.
 Then since $(\pi \otimes \pi) \circ \Delta(e_{11}) =
\Delta \circ \pi(e_{11})= \Delta(t) \in T_q^{[0]} \otimes T_q^{[0]} $,
  and since for $e_{ij} \in D^{[0]}$,
 $i \neq j$, each
$\pi(e_{ij})$ is a scalar multiple of $ f_{\lambda}$
for some primitive $p$th root of unity $\lambda$, we have that
\begin{equation}\label{eqn: delta(t)}
\Delta(t) = t \otimes t + \beta_2 f_{\lambda_2} \otimes
f_{\lambda_2^\prime} + \ldots +
\beta_n f_{\lambda_n} \otimes   f_{\lambda_n^\prime}.
\end{equation}
By comparing coefficients of $1 \otimes 1$ on both
sides of  (\ref{eqn: delta(t)}) we obtain:
\begin{displaymath}
 1 = 1 + \beta_2 + \beta_3 + \ldots + \beta_n,
\end{displaymath}
and comparing coefficients of $g^i \otimes 1$ for $i >0$
on both sides of  (\ref{eqn: delta(t)}) we get:
\begin{displaymath}
 0 = 1 + \beta_2 \lambda_2^{i} + \beta_3 \lambda_3^{i}
+ \ldots + \beta_n \lambda_n^{i}.
\end{displaymath}
Then, adding the coefficients of   $g^i \otimes 1$ for
$0 \leq i \leq p-1$
on both sides of  (\ref{eqn: delta(t)}),  we have:
\begin{displaymath}
 1 =  p + \beta_2[ 1 + \lambda_2^{1} + \lambda_2^{2}
+ \ldots + \lambda_2^{(p-1)}] + \ldots
 + \beta_n[ 1 + \lambda_n^{1} + \lambda_n^{2}
+ \ldots + \lambda_n^{(p-1)}] = p,
\end{displaymath}
a contradiction. Thus $L_c$ properly permutes either the
$R_i$ or the $M_i$ and the proof is complete. \end{proof}

The next proposition  shows that Hopf algebras of dimension
$p^3$ of type $(p,p)$ where both $H$ and $H^\ast$
have nontrivial skew-primitives cannot have  simple
subcoalgebras of dimension $p^2$.

\begin{prop}\label{prop: no simple D of dim p squared}
Assume $H$ has a pointed sub-Hopf algebra
$T \cong T_{q^\prime}$  and
there is a Hopf algebra projection
$\pi: H \twoheadrightarrow T_q = \Bbbk \langle g,y \rangle$.
Then $H$ has no
simple subcoalgebra of dimension $p^2$.
\end{prop}

\begin{proof}
Let $c \in G(H)$ be such that $\pi(c) = g$.
Let $D$ be a simple
subcoalgebra of $H$ of dimension $p^2$, and we seek a
contradiction.  As in the proof of Lemma \ref{lemma: permutes},
let $T_q^{[i]}:= \{ g^jy^i | 0 \leq j \leq p-1 \}$ and $ D^{[i]}$
denote the eigenspaces for the eigenvalue $q^i$
for the coalgebra morphisms $\adl{g}$ and $\adl{c}$ in $T_q$ and $D$ respectively.
 Let $\mathbf{e}$ be a multiplicative matrix
 for the coalgebra isomorphism $\adl{c}$
of  $D$ and let  $R_i \subset D$ be
 the simple right $D$-comodule with basis $\{ e_{ij} | 1 \leq j \leq p \}$.

Let $A = (a_{ij}) = (\omega_i \omega_j^{-1})$ be the
$p \times p$ matrix whose entries are the eigenvalues
for the multiplicative matrix $\mathbf{e}$ from Proposition \ref{prop:gen-stefan-vay}.
 Since by
Proposition \ref{prop:gen-stefan-vay} the $a_{ij}$ are $p$th roots
of unity,  each $a_{ij} = q^k$ for some $k$.  First we wish
to show that each row of $A$ contains the $p$ distinct entries
$\{ 1, q, \ldots, q^{p-1} \}$ in some order so that the
dimension of $D^{[i]}$ is $p$ for every $i$.
 By Lemma \ref{lemma: permutes}, $L_c$ permutes either
the simple left $D$-subcomodules $M_i$ or the simple
right $D$-subcomodules $R_i$ nontrivially.
We assume that $L_c$ permutes the $R_i$ nontrivially;
the argument is
the same if the $M_i$ are permuted nontrivially.
 Thus each row of $A$ contains exactly the same entries,
perhaps in a different order.
We note that each row of $A$ contains an entry $1$
since $a_{ii}=1$ and also must contain an entry $q$
since $\pi(D)$ generates
$T_q$ as an algebra.
\smallbreak
  Let $N$ be the maximum number of equal entries
in a row and suppose that $a_{1 i_1} = a_{1 i_2} =
\ldots = a_{1 i_N}$. Then $\omega_{i_1} = \omega_{i_2}
= \ldots = \omega_{i_N}$
so that rows $\omega_{i_1}, \ldots, \omega_{i_N}$
are equal as vectors in $\Bbbk^p$ and the columns
 $\omega_{i_1}, \ldots, \omega_{i_N}$
are also equal as  vectors in $\Bbbk^p$.  Thus $1 =
 a_{i_1,i_1} = a_{i_1,i_2} =  \ldots, a_{i_1,i_N}$ and so
 since row $i_1$ has $N$ entries equal to $1$, each row of $A$
also contains
$N$ entries equal to $1$.  By the maximality of $N$,
the rows must contain exactly $N$ entries equal to $1$.

 \smallbreak Let $ a_{i_1 k}$ be an entry in row $i_1$
different from $1$.  Then $a_{i_1 k}^{-1} = a_{k i_1}$ lies in column
 $i_1$ and since columns $i_1, \ldots, i_N$ are equal,
the $k$th row contains $N$ entries equal to $a_{k i_1}$.
Thus row $i_1$ also
 has $N$ entries equal to $a_{ki_1}$.
The same argument with entry $a_{ki_1}$ in row $i_1$
shows that row $i_1$ has $N$ entries
 equal to $a_{ki_1}^{-1} = a_{i_1 k}$.
Thus every distinct entry in row $i_1$ occurs exactly $N$
times and so  $N$ divides $p$. Thus $N=p$ or $N=1$.
 If $N=p$ then all entries in $A$ equal $1$ contradicting
the fact that $\pi(D)$ generates $T_q$. Thus $N=1$.

Thus each $R_i$ contains precisely one eigenvector for each
eigenvalue $q^i$. By relabelling if necessary we may
assume that $e_{1i}$ is an eigenvector for $q^{i-1}$
and then $e_{i1}$ is an eigenvector for
 $q^{1-i}$.  Then $e_{jk}$ is an eigenvector for
$q^{1-j}q^{k-1} = q^{k-j}$ and
 $D^{[i]} = \{ c^je_{1,i+1} | 0 \leq j \leq p-1 \}$.

 Let $ 0 \leq i \leq p-1$. Suppose $0 \neq \pi( e_{1,i+1})= \phi(g)y^i$
where $\phi(g) = a_0 + a_1g + \ldots a_{p-1}g^{p-1}$.
Suppose that $a_k \neq 0$. Then $g^k y^i \otimes g^k \in \pi(D) \otimes \pi(D)$
 and so $g^k y^i \in \pi(D)$.
Since $\pi(D)$ is invariant under left multiplication by $g$,
this means that all $g^j y^i \in \pi(D)$ and so
if the dimension of $\pi(D^{[i]})$ is nonzero, it is $p$.
Note that the dimensions of $\pi(D^{[0]})$ and $\pi(D^{[1]})$
must be $p$ and that  if the dimension of $\pi(D^{[i]})$ is $p$,
then the dimension of $\pi(D^{[j]})$ is also $p$ for $j< i$.

Let $M$ be the maximum such that $\pi(D^{[M]}) \neq 0$
so that $1 \leq M \leq p-2$ and then $0 = \pi(e_{1, n})$
for $n > M+1$.
Then \begin{displaymath}
0 = \Delta(\pi(e_{1, M+1})) = \sum_{i=1}^M \pi(e_{1, i })
\otimes \pi(e_{i, M+1}) \in \sum \pi(D^{[i-1]}) \otimes \pi(D^{[M+1-i]}).
\end{displaymath}
But if any of the terms $ \pi(e_{1, i })$ or $\pi(e_{i, M+1})$
is $0$ then the dimension of
 $\pi(D^{[i-1]})$ or $ \pi(D^{[M+1-i]})$ is less than $p$,
so we have a  contradiction.
\end{proof}

 \begin{cor}\label{cor: (3,3) no 9 dim subcoalgebras}
  If $p=3$ and $H$ is of type $(3,3)$, then $H$ has
no subcoalgebra $\M^\ast(3,\Bbbk)$.
 \end{cor}

 \begin{proof} By   Proposition \ref{prop:biti-dasca}$(ii)$, both
$H$ and $H^\ast$ contain a $9$-dimensional
Taft Hopf algebra. Now apply Proposition
\ref{prop: no simple D of dim p squared}.
 \end{proof}

 \begin{cor} \label{cor: dual no sps}
Let $H$ have grouplikes  of order $p$ and
  suppose $H_0 \cong \Bbbk G(H) \oplus \mathcal{M}^\ast(p,\Bbbk)^t$.
  Then  if  $t \geq p-3$, $H^\ast$ has no skew-primitives.
  \end{cor}

\begin{proof}
 By Remark \ref{rem: has sps}, $H$ has a skew-primitive element
and so has a sub-Hopf algebra isomorphic to
 a Taft Hopf algebra.  If $H^\ast$ also had a skew-primitive element,
there would be a contradiction to
 Proposition \ref{prop: no simple D of dim p squared}.
\end{proof}

\begin{example}\label{ex:5-5} Suppose $H$ has dimension $5^3$,
and is of type $(5,5)$.
Then $H$ and $H^\ast$ cannot both
have a skew-primitive element.

For suppose that   $H$ contains a sub-Hopf algebra
$T_{q^\prime} $
 and there is also a Hopf algebra projection $\pi$ from $H$ to
$T_q = \Bbbk \langle g, y \rangle$. Then by
Proposition \ref{prop: no simple 4 dim}, $H$ has no simple
subcoalgebra of dimension $2^2$, by Corollary \ref{cor: no dim 3 squared},
$H$ has no simple
subcoalgebra of dimension $3^2$, and  by Proposition
\ref{prop: no simple D of dim p squared}, $H$ has no simple
subcoalgebra  of dimension $5^2$.
Thus if $H$ has any simple subcoalgebra  of dimension
greater than $1$ it must have dimension $4^2$,
\textit{i.e.}, $H_0 \cong \Bbbk C_5 \oplus \M^\ast(4, \Bbbk)^{tp}$ and
this is impossible by Corollary \ref{cor: H0 not kCp + p(p-1)sq}.
 \end{example}

\subsection{Hopf algebras of dimension $27$}\label{subsect: 27}

In this short subsection we  complete the classification for dimension $27$.

\begin{theorem}\label{thm dim 27}
A Hopf algebra of dimension $27$ is semisimple, pointed or copointed.
\end{theorem}

\begin{proof} Assume $H$ is nonsemisimple, nonpointed and
non-copointed. Then by \cite{GG} we need only
consider types $(3,3)$ and $(3,1)$. First consider type $(3,3)$.
 By Proposition \ref{prop:biti-dasca}, if $H$ is of type
$(3,3)$ then $H$ and $H^\ast$
 both contain a Taft Hopf algebra of dimension $9$.
By Corollary \ref{cor: (3,3) no 9 dim subcoalgebras},
$H$ has no subcoalgebra isomorphic to $\M^\ast(3,\Bbbk)$ and
by Proposition \ref{prop: no simple 4 dim},
$H$ has no subcoalgebra isomorphic to $\M^\ast(2, \Bbbk)$,
which implies that $H_{0} = \Bbbk G(H)$, a contradiction.

To complete the proof, we show that no Hopf algebra
of dimension $27$ can have only trivial grouplikes.
If $|G(H)| =1$, all possible coradicals are listed in the
table below. We show that each leads to a contradiction.

\begin{table}[here]
\begin{center}
\tiny{\begin{tabular} {|c|l|c|} \hline
{\bf Case} & {\bf $\Ho$} & {\bf $\dim \Ho$} \\
\hline $(i)$ & $\Bbbk\cdot 1\oplus \coM^{n}$, $1\leq n \leq 6$ & $1
+ 4n$\\ \hline $(ii)$ & $\Bbbk\cdot 1\oplus \coMtres^{n}$, $1\leq n
\leq 2$ & $1 + 9n$ \\ \hline $(iii)$ & $\Bbbk\cdot 1\oplus \coMcua$
& $1 + 16$ \\ \hline $(iv)$ & $\Bbbk\cdot 1\oplus \coMcin$ & $1 +
25$ \\ \hline $(v)$ & $\Bbbk\cdot 1\oplus \coM^{n}\oplus
\coMtres^{m}$, $0< n, m,\ 0< 4n + 9m < 26$ & $1 + 4n + 9m$ \\ \hline
$(vi)$ & $\Bbbk\cdot 1\oplus \coM^{n}\oplus \coMtres^{m}\oplus
\coMcua$, $0\leq n, m,\ 0< 4n + 9m +16 < 26$ & $1 + 4n + 9m + 16$ \\
\hline
\end{tabular}}
\end{center}
\ \caption{ Coradicals  $\dim H = 27$, $|\GH| =
1$}\label{tab-coradicals}
\end{table}

\par Case $(i)$ is not possible by
Example \ref{example: dim 27 no simples dim 4}.
 Case $(ii)$ is impossible by Example \ref{ex: dim}.
Cases $(iii)$, $(iv)$   are impossible by
Proposition \ref{prop:biti-dasca}$(ii)$.

\par  Only cases $(v)$, $(vi)$ remain.
By Proposition \ref{prop: dim simples dim 4}, $H$ cannot have
a simple $4$-dimensional subcoalgebra stable
under the antipode so the case $n=1$ is impossible in either $(v)$ or $(vi)$.
The remainder of case $(vi)$ is impossible by
Proposition \ref{prop:biti-dasca}.
\par The last remaining case is coradical
$H_0 \cong \Bbbk \cdot 1 \oplus
\M^\ast(2, \Bbbk)^2 \oplus \M^\ast(3,\Bbbk)$ since
one sees immediately that if $n>2$ or $m>1$
then the dimension of $H_0$ is impossibly big
and Proposition \ref{prop:biti-dasca}
will give a contradiction. So suppose that
$H_0 \cong \Bbbk \cdot 1 \oplus  D \oplus E \oplus \M^\ast(3,\Bbbk)$
where $D \cong \M^\ast(2,\Bbbk)$, $E = S(D)$ and $S^2(D) = D$; in particular
$H_0$ has dimension $18$.
If $P^{- ,\M^\ast(3,\Bbbk)} \neq 0$ then the usual
dimension arguments give a contradiction.
So suppose that these spaces are $0$. By Proposition \ref{prop:biti-dasca}$(i)$,
$P^{1,1}$ is nonzero and this implies by Lemma \ref{lema:fukuda}$(i)$ that
$P^{1,D},P^{D,1}$ are nonzero.
But then $P^{1,E},P^{E,1}$ are nonzero also
and so are $P^{D,X},P^{E,Y}$ where $X,Y \in \{ D,E \}$.
Then the dimension of $H$ is at least $18 +1 +  8 + 8 = 35$, a contradiction.
\end{proof}

We end this note by giving the complete
list of Hopf algebras of dimension $27$.

\begin{obs}\label{rmk:class27}
By Theorem \ref{thm dim 27}, if
$ H $ is a Hopf algebra of dimension $ 27 $
then $ H $ is either semisimple or pointed or
the dual is pointed.
By the classification of semisimple and pointed Hopf algebras of
dimension $ p^{3} $, $ H $
 is isomorphic to exactly one  Hopf algebra
in the following list:

Semisimple Hopf algebras of dimension $27$ were classified by Masuoka
\cite{ma-pp};
there
are $11$ isomorphism types, namely
\begin{enumerate}
\item[$(a)$] Three group algebras of abelian groups.
\item[$(b)$] Two group algebras of nonabelian groups, and their duals.
\item[$(c)$] $4$ self-dual Hopf algebras which are neither commutative nor
cocommutative. They are extensions of
$\Bbbk[C_{3}\times C_{3}]$ by $\Bbbk C_{3}$.
\end{enumerate}

Pointed Hopf algebras  of dimension $27$ were classified
independently by different authors, see
\cite{AS2}, \cite{CD} and \cite{Stv}.
Here $q $ is a primitive $3$-rd root of unity
and $ T_{q} $ the Taft Hopf algebra of dimension $ 9 $.
 Note that in $(d)$
 the grouplikes are isomorphic to
$C_3 \times C_3$, in $(e),(f),(g)$ to $C_9$, in $(h),(i)$ to $C_3$ and $(j),(k)$
are copointed but not pointed.
\begin{enumerate}
\item[$(d)$] The tensor-product Hopf algebra $T_{q} \otimes \Bbbk C_{3}$.
\item[$(e)$] $\widetilde{T_{q}}  := \Bbbk\langle g,\ x |\ gxg^{-1} =
q^{1/3}x,\ g^{9} = 1,\ x^{3} = 0\rangle$ ($q^{1/3}$ a $3$-th
root of $q$), with comultiplication
$\Delta (x) = x\otimes g^{3} + 1\otimes x,\ \Delta (g) = g\otimes g$.
\item[$(f)$] $\widehat{T_{q}} := \Bbbk\langle g,\ x |\ gxg^{-1} =
qx,\ g^{9} = 1,\ x^{3} = 0\rangle$, with comultiplication
$\Delta (x) = x\otimes g + 1\otimes x,\ \Delta (g) = g\otimes g$.
\item[$(g)$] ${\bf r}(q) := \Bbbk\langle g,\ x |\ gxg^{-1} = qx,\
g^{9} = 1,\ x^{3} = 1 - g^{3}\rangle$, with comultiplication
$\Delta (x) = x\otimes g + 1\otimes x,\ \Delta (g) = g\otimes g$.
\smallbreak
\item[$(h)$] The Frobenius-Lusztig kernel ${\bf u}_{q}({\mathfrak {sl}}_{2})
:= \Bbbk\langle g,\ x,\ y |\ gxg^{-1} = q^{2}x,$
$gyg^{-1} = q^{-2}y$, $g^{3} = 1$, $x^{3} = 0$,
$y^{3} = 0$, $xy - yx = g - g^{-1}\rangle$, with comultiplication
$\Delta (x) = x\otimes g + 1\otimes x,\ \Delta (y) =
y\otimes 1 + g^{-1}\otimes y,\
\Delta (g) = g\otimes g$.
\item[$(i)$] The book Hopf algebra ${\bf h}
(q, m) := \Bbbk\langle g,\ x,\ y |\ gxg^{-1} = qx,\ gyg^{-1} = q^{m}y,\
g^{3} = 1,\ x^{3} = 0,\ y^{3} = 0,\ xy - yx = 0\rangle,\ m \in C_{3}
\smallsetminus \{0\}$,
with comultiplication
$\Delta (x) = x\otimes g +
1\otimes x,\ \Delta (y) = y\otimes 1 + g^{m}\otimes y,\
\Delta (g) = g\otimes g$.
\item[$(j)$] The dual of the Frobenius-Lusztig kernel,
${\bf u}_{q}({\mathfrak {sl}}_{2})^{*}$.
\item[$(k)$] The dual of the case $(g)$, ${\bf r}(q)^{*}$.
\end{enumerate}
\end{obs}

\section*{Acknowledgements}
This work was begun when G.A.G was visiting Mount
Allison University. He wants to thank all the
members of the Math/CS department
for their warm hospitality. The authors
also thank C. Vay for providing the proof of
Proposition \ref{prop:gen-stefan-vay}.

%

\end{document}